%% file: paper1.tex
\newcommand{\papertitle}{Conservative Three-Layer Schemes for Kirchhoff-Type
Equations}

\newcommand{\paperkeywords}{Conservative discretization; Kirchhoff-type equation;
nonlinear wave equation; three-layer scheme; energy conservation; stability and convergence}

\newcommand{\papermsc}{65M06; 65M12; 65M15; 35L70; 35L90}

\newcommand{\paperabstract}{%
We study a Kirchhoff-type nonlinear integro-differential equation in two spatial
dimensions whose coefficients are allowed to depend on time, and we construct
conservative discretizations in time for the associated initial--boundary value
problem. We consider two symmetric three-layer schemes of Crank--Nicolson type, a
locally linear one and a genuinely nonlinear one, and we show that each of them
preserves a discrete analogue of the total mechanical energy of the homogeneous
problem with constant coefficients. For the nonlinear scheme we establish uniform
\apriori{} bounds on the discrete solution and on its discrete velocity by working
directly with the discrete energies, without invoking a nonlinear discrete Gr\"onwall
inequality, the constants still grow exponentially in the final time, and we prove local second-order
convergence in time, both for the solution and for the central-difference
approximation of its first time derivative. The nonlinear system arising at each
time level is solved by a fixed-point iteration: given iterates at the two preceding
levels that satisfy the \apriori{} bounds, that equation has exactly one solution and
the iteration converges to it at a geometric rate once the time step is small enough.
Since the \apriori{} bounds are index-local, alternating them with that one-step solver
constructs the trajectory stepwise, for time-dependent coefficients as well, on the
local interval on which those bounds hold. Numerical experiments, carried out in a setting in which the spatial
discretization contributes no error, exhibit the conservation of the discrete
invariants, confirm the second order in time and verify the geometric convergence of the
fixed-point iteration.%
}

\newcommand{\paperdataavailability}{%
The data underlying every figure and table of \Cref{sec:numerics}, together with the sources of
all figures, are contained in the \texttt{figures/} directory of the accompanying repository,
\REPOSITORYURL. 
}

\newcommand{\papersubject}{Numerical analysis of conservative three-layer time discretizations for Kirchhoff-type integro-differential equations with time-dependent coefficients}

\input{common/preamble}

\begin{document}

\input{common/titlepage}

\input{paper1-body}

\input{common/backmatter}

\end{document}

%% file: common/preamble.tex
\documentclass[a4paper]{article} % A4 stock; 'article' is standard for papers/notes

\newif\ifsubmission  \submissiontrue
\newif\ifblind       \blindfalse

\usepackage[
a4paper,                          % A4 stock
margin=28mm,                      % Uniform text margins
marginparwidth=22mm,              % Wide enough for todonotes
includehead, includefoot          % Count header/footer in text area (optional)
]{geometry}

\usepackage{xcolor}                 % Colors for text, rules, links, etc.

\usepackage{multicol}               % \begin{multicols}{2} ... \end{multicols}

\usepackage{longtable,booktabs}     % Long tables + high-quality horizontal rules

\usepackage{graphicx}               % \includegraphics, \graphicspath
\graphicspath{{figures/}}           % (Optional) default path for figures

\usepackage[section]{placeins}
\usepackage{standalone}
\usepackage{tikz}
\usepackage{pgfplots}
\pgfplotsset{compat=1.18}

\usepackage{amsmath}                % Math environments (align, gather, etc.)
\usepackage{amssymb}                % Math symbols (\mathbb, \mathfrak, etc.)
\usepackage{amsthm}                 % \newtheorem environments
\usepackage{mathtools}              % (Optional) amsmath enhancements (\coloneqq, \DeclarePairedDelimiter, ...)
\usepackage{mathrsfs}

	\usepackage{silence}                % Selectively suppress warnings
	\usepackage[T1]{fontenc}                 % Proper glyphs & hyphenation in PDF
	\usepackage[utf8]{inputenc}              % UTF-8 source for pdfLaTeX
	\usepackage{lmodern}                     % Latin Modern fonts (vector)
	\usepackage[nopatch=footnote]{microtype} % Improves justification; no footnote patch (avoid conflicts)
	
	\usepackage{ragged2e}
	
	\usepackage[normalem]{ulem}         % \sout{}, but keep \emph behavior normal
	
	\usepackage{titlesec}               % Custom section formats
	\usepackage{setspace}               % \onehalfspacing, \doublespacing etc.
	\usepackage{indentfirst}            % Indent first paragraph of each section
	
	\usepackage[inline]{enumitem}
	
	\newcommand{\authorKrumbiegel}{Felix Krumbiegel}
    \newcommand{\authorRogava}{Jemal Rogava}
	\newcommand{\authorRupp}{Andreas Rupp}
	\newcommand{\authorVashakidze}{Zurab Vashakidze}
	
	\newcommand{\orcidKrumbiegel}{0009-0002-7288-4071}
    \newcommand{\orcidRogava}{0000-0001-9460-4283}
	\newcommand{\orcidRupp}{0000-0001-5527-7187}
	\newcommand{\orcidVashakidze}{0000-0001-8736-6213}
	
	\newcommand{\paperauthors}{\authorKrumbiegel{}; \authorRogava{}; \authorRupp{}; \authorVashakidze{}}

	\newif\ifrepositoryresolved  \repositoryresolvedtrue
	\ifrepositoryresolved
		\newcommand{\REPOSITORYURL}{\url{https://github.com/FelixKrumbiegel/Kirchhoff-conservative-FEM}}
	\else
		\ifblind
			\newcommand{\REPOSITORYURL}{\texttt{[repository URL withheld for double-blind review]}}
		\else
			\newcommand{\REPOSITORYURL}{\url{https://github.com/FelixKrumbiegel/Kirchhoff-conservative-FEM}}
		\fi
		\AtEndDocument{%
			\GenericWarning{}{%
				*** SUBMISSION BLOCKER: the Data and Code Availability statement still contains a
				placeholder. Set \string\repositoryresolvedtrue\space and fill in
				\string\REPOSITORYURL\space with the persistent archive DOI before submitting ***%
			}%
		}
	\fi
	
	\ifblind
		\newcommand{\runauthors}{\runningtitle}
	\else
		\newcommand{\runauthors}{\authorKrumbiegel{} \textbullet\ \authorRogava{} \textbullet\ \authorRupp{} \textbullet\ \authorVashakidze{}}
	\fi
	\newcommand{\runningtitle}{\papertitle} % Adjust to a shorter variant if needed
	
	\usepackage{fancyhdr}               % Custom header/footer management
	\AtBeginDocument{\thispagestyle{empty}} % Also set in titlepage below; redundancy is harmless
	
	\fancypagestyle{plain}{%
		\fancyhf{}
		\fancyhead[C]{%
			\ifodd\value{page}%
			\small \runauthors%
			\else%
			\small \runningtitle%
			\fi%
		}
		\fancyfoot[C]{\thepage}

	}
	
	\definecolor{linkColor}{RGB}{155,0,119} % Custom link color (accessible purple)
	\ifblind
		\newcommand{\pdfauthorfield}{}
	\else
		\newcommand{\pdfauthorfield}{\paperauthors}
	\fi
	\usepackage[
	pdfstartview  = {FitH},          % Open PDF with page width fitting window
	pdftitle      = {\papertitle},   % PDF metadata: title (from macro)
	pdfauthor     = {\pdfauthorfield}, % PDF metadata: author(s) (blank if \blindtrue)
	pdfkeywords   = {\paperkeywords},  % PDF metadata: keywords (was empty)
	pdfsubject    = {{\papersubject}},   % PDF metadata: subject (braced: the value may contain commas)
	colorlinks    = true,            % Colorize links instead of boxed links
	linkcolor     = linkColor,       % Section/cross-ref link color
	citecolor     = linkColor,       % Citation link color
	urlcolor      = linkColor,       % URL link color
	hypertexnames = false            % Avoid duplicate anchors (safer TOC/refs/cross-refs)
	]{hyperref}                         % Load hyperref once with options
	\usepackage{orcidlink} % Official ORCID iD mark, works with pdfLaTeX/XeLaTeX/LuaLaTeX
	
	\usepackage[nameinlink]{cleveref}   % Smart cross-references; auto “Fig.”, “Sec.”, etc.
	\crefname{equation}{eq.}{eqs.}      % "eq. (1)" / "eqs. (1) and (2)"
	\Crefname{equation}{Eq.}{Eqs.}      % "Eq. (1)" at sentence start
	\crefname{theorem}{theorem}{theorems}
	\Crefname{theorem}{Theorem}{Theorems}
	\crefname{figure}{figure}{figures}
	\Crefname{figure}{Figure}{Figures}
	\crefname{table}{table}{tables}
	\Crefname{table}{Table}{Tables}
	
	\theoremstyle{plain}
	
	\newtheorem{theorem}{Theorem}[section]      % Title: Theorem (numbered by section)
    \crefname{theorem}{Theorem}{Theorems}   
	
	\newtheorem*{theorem*}{Theorem}             % Title: Theorem* (unnumbered)
	
	\newtheorem{lemma}[theorem]{Lemma}          % Title: Lemma (shares counter)
	
	\newtheorem*{lemma*}{Lemma}                 % Title: Lemma* (unnumbered)
	
	\newtheorem{corollary}[theorem]{Corollary}  % Title: Corollary (shares counter)
    \crefname{corollary}{Corollary}{Corollaries}   
	
	\newtheorem*{corollary*}{Corollary}         % Title: Corollary* (unnumbered)
	
	\theoremstyle{definition}
	
	\newtheorem*{definition*}{Definition}        % Title: Definition* (unnumbered)
	
	\newtheoremstyle{boldremark} % Name
	{3pt}                      % Space above: vertical skip before the environment
	{3pt}                      % Space below: vertical skip after the environment
	{\normalfont}              % Body font: upright/roman text (not italic)
	{}                         % Indent amount: empty = no indentation
	{\bfseries}                % Head font: bold heading (e.g., "Remark 1.2.")
	{.}                        % Punctuation after head: prints a period after heading
	{.5em}                     % Space after head: spacing between head and body text
	{}                         % Head specification: empty = default layout
	
	\theoremstyle{boldremark}
	
	\newtheorem{remark}[theorem]{Remark}        % Title: Remark
	
	\newtheorem*{remark*}{Remark}               % Title: Remark* (unnumbered)
	
	\theoremstyle{definition}

	\theoremstyle{plain}
	
	\theoremstyle{definition}
	
	\newcommand{\refitem}[2]{%
		\hyperref[#2]{#1~\ref*{#2}}% \ref* prints the bare number; \hyperref makes it a single link
	}
	
	\newcommand{\figsubref}[2]{%
		\hyperref[#2]{Figure~\ref*{#1}(\subref*{#2})}% one clickable "Figure <main>(<sub>)"
	}
	
	\newcommand{\lemref}[1]{%
		\refitem{Lemma}{#1}%
	}
	
	\newcommand{\thmref}[1]{%
		\refitem{Theorem}{#1}%
	}
	
	\newcommand{\corref}[1]{%
		\refitem{Corollary}{#1}%
	}

	\makeatletter
	\renewenvironment{proof}[1][\proofname]{%    % Title: proof environment (redefinition)
		\par\pushQED{\qed}\normalfont%           % Push end-of-proof symbol; normal font
		\topsep6\p@\@plus6\p@\relax%             % Vertical spacing before proof
		\trivlist\item[\hskip\labelsep\bfseries#1\@addpunct{.}]% % Bold "Proof." label with punctuation
		\ignorespaces%                           % Ignore spaces after the label
	}{%
		\popQED\endtrivlist\@endpefalse%         % Pop QED symbol; end list; reset end-paragraph flag
	}
	\makeatother
	
	\newcommand{\R}{\mathbb{R}}         % Example: real numbers symbol
    \newcommand{\domain}{\Omega}

	\long\def\dd{\mathrm{d}}           % Upright differential 'd' for dx, dy, etc.
	\long\def\bigO{\mathcal{O}}        % Big-O notation
	\long\def\I{\mathcal{I}}           % Calligraphic I
	\long\def\L{\mathcal{L}}           % Calligraphic L (shadows LaTeX's Polish L; see \Lslash above)
	\long\def\T{\mathcal{T}}           % Calligraphic T
	\long\def\ie{\textit{i.e.}}        % Latinism: id est
	\long\def\eg{\textit{e.g.}}        % Latinism: exempli gratia
	\long\def\cf{\textit{cf.}}         % Latinism: confer
	\long\def\apriori{\textit{a priori}} % Latinism: a priori
	\newcommand{\dint}[2][\Omega]{%    % Title: \dint (double integral macro)
		\iint\limits_{#1}%             % Place domain under the integral in display math
		#2%                            % Integrand (function under the integral sign)
		\,\mathrm{d}x%                 % Thin space + upright differential 'd' for x
		\,\mathrm{d}y%                 % Thin space + upright differential 'd' for y
	}
	
	\input{constants_enumeration}
	
	\titleformat{\section}{\large\bfseries}{\thesection}{1em}{}            % Title: \section format (bold, larger)
	\titleformat{\subsection}{\normalsize\bfseries}{\thesubsection}{1em}{} % Title: \subsection format (bold)
	
	\titlespacing*{\section}{0pt}{*2}{*0.6}
	\titlespacing*{\subsection}{0pt}{*1.3}{*0.4}
	
	\usepackage[%
                doi=false,
                backend=biber,
                style=numeric,
                ]{biblatex} % Use biblatex with biber; style = alphabetic (e.g. [KnA21])
	\appto\biburlsetup{\Urlmuskip=0mu plus 1mu\relax}
	\numberwithin{equation}{section} % Equations: (section.eqnum)
	\numberwithin{figure}{section}   % Figures: (section.fignum)
	\numberwithin{table}{section}    % Tables: (section.tabnum)
	
	\usepackage[most]{tcolorbox} % Load tcolorbox with common libraries (skins, breakable, theorems)
	\tcbset{%
		myabstract/.style={%
			colback=gray!5!white,    % Light background
			colframe=black!40,       % Frame color
			fonttitle=\bfseries,     % Bold title
			coltitle=black,          % Title text color
			boxrule=0.6pt,           % Border thickness
			arc=2mm,                 % Rounded corners
			enhanced,                % Enable advanced options
			before skip=10pt,        % Space before the box
			after skip=20pt,         % Space after the box
			breakable                % Allow page breaks if the abstract is long
		}%
	}
	\renewenvironment{abstract}%      % Title: abstract environment (redefinition)
	{\begin{tcolorbox}[myabstract,title=Abstract]}%  % Begin styled abstract box
		{\end{tcolorbox}}%                           % End styled abstract box
	
	\title{\papertitle}                 % For PDF metadata/bookmarks
	\author{\runauthors}                % For PDF metadata/bookmarks
	\date{}                             % Empty date suppresses date printing under the title
	
	\newcommand{\titleseparator}{\par\noindent{\color{black}\rule{\linewidth}{0.8pt}}\par}
	
	\allowdisplaybreaks

	\AtBeginDocument{%
		\setlength{\abovedisplayskip}{5pt plus 2pt minus 2pt}%
		\setlength{\belowdisplayskip}{5pt plus 2pt minus 2pt}%
		\setlength{\abovedisplayshortskip}{2pt plus 1pt minus 1pt}%
		\setlength{\belowdisplayshortskip}{3pt plus 1pt minus 1pt}}
	\makeatletter\def\thm@space@setup{\thm@preskip=4pt \thm@postskip=4pt}\makeatother  % Allow multi-line displayed equations to break across pages

%% file: constants_enumeration.tex
\newcounter{constant}

\newcounter{mconst}

\newcommand{\cst}[1]{%
  \refstepcounter{constant}% Increase 'constant' and make it referenceable
  \expandafter\xdef\csname #1\endcsname{c_{\arabic{constant}}}%
}

\newcommand{\mcst}[1]{%
  \refstepcounter{mconst}% Increase 'mconst' and make it referenceable
  \expandafter\xdef\csname #1\endcsname{M_{\arabic{mconst}}}%
}

\mcst{cstmone}    % Defines \cstmone   -> M_1  (used in Lemma 1)
\mcst{cstmtwo}    % Defines \cstmtwo   -> M_2  (used in Lemma 1)
\mcst{cstmthree}  % Defines \cstmthree -> M_3  (used in Lemma 2)
\mcst{cstmfour}   % Defines \cstmfour  -> M_4  (used in Lemma 2)

\mcst{cstmfive}   % Defines \cstmfive  -> M_5  (used in the Iteration Theorem)
\mcst{cstmseven}  % Defines \cstmseven -> M_6  (used in the Iteration Theorem; the macro name
\mcst{cstmseight}  % Defines \cstmseight -> M_7  (the macro name carries a stray "s" and the
\mcst{cstmsix}    % Defines \cstmsix -> the NEXT free M index.  Used only by main.tex;
\cst{lmonecstalph}   % Defines a c-type constant associated with alpha in Lemma 1
\cst{lmonecstbeta}   % Defines a c-type constant associated with beta in Lemma 1
\cst{lmonecstdiff}   % Defines a c-type constant for difference estimates in Lemma 1

\cst{lmtwocstdiffbw}             % Defines a c-type constant for differences in Lemma 2
\cst{lmtwocsthatenerg}           % Defines a c-type constant for the hat-energy term in Lemma 2
\cst{lmtwocstfracone}            % Defines a c-type constant for the first fractional bound in Lemma 2
\cst{lmtwocstfractwo}            % Defines a c-type constant for the second fractional bound in Lemma 2
\cst{lmtwocstmaxf}               % Defines a c-type constant controlling a maximum function in Lemma 2
\cst{lmtwocstnearestexp}         % Defines a c-type constant for nearest-exponent estimates in Lemma 2
\cst{lmtwocstnearestfinineq}     % Defines a c-type constant for final inequality bounds in Lemma 2

\cst{cststatementone}            % Defines a c-type constant appearing in the statement of Theorem 1
\cst{cststatementtwo}            % Defines a second c-type constant appearing in the statement of Theorem 1
\cst{thmcstdiffbetaw}            % Defines a constant in the difference of beta times double-u in Theorem 1
\cst{thmcsthalflnorm}            % Defines a constant in the estimate of norm of L^{1/2} applied to solution u in Theorem 1
\cst{thmcstlnorm}                % Defines a constant in the estimate of the norm of L applied to the solution u in Theorem 1
\cst{thmcstinnerprodgk}          % Defines a constant in the upper bound of the absolute value of the inner product associated with g_k in Theorem 1
\cst{thmcstinnerprodrem}         % Defines a constant in the upper bound of the absolute value of the inner product associated with the remainder term in Theorem 1
\cst{thmcstsumoverczero}         % Defines a constant in the upper bound in E_{k + 1} as an fraction of some constants
\cst{thmcstmaxtwoconst}          % Defines a constant in the upper bound in E_{k + 1} as an maximum between two constants
\cst{thmcstesttone}              % Defines a constant in the upper bound in \tilde{T}_1
\cst{thmcstestwone}              % Defines a constant in the upper bound in \tilde{W}_1
\cst{thmtildevareps}             % Defines a constant in the upper bound in \tilde{\varepsilon}
\cst{thmtildeEkplusone}          % Defines a constant in the upper bound in \sqrt{\tilde{E}_{k + 1}}

\cst{rmkthmcstkinetic}           % Defines a constant in the upper bound in | T \left( t_k \right) - T_k^{\ast} |

\cst{thmcstesttildc}             % upper bound for \tilde{\varepsilon} (main.tex only)
\cst{thmcstsqrtekplusone}        % upper bound for \sqrt{\tilde{E}_{k+1}} (main.tex only)

\cst{cstczero}    % generic c-type constant, alpha lower bound (legacy)
\cst{cstcone}     % generic c-type constant, beta lower bound (legacy)

%% file: common/titlepage.tex
% Shared title page. Uses \papertitle, \paperabstract, \paperkeywords, \papermsc.
		
		\sloppy  % Relax line-breaking rules globally (looser spacing, fewer overfull boxes)
		
		% ======================================================================
        %  CUSTOM TITLE PAGE
        %  Title: Manual title-page layout
        %  Scope:
        %    - Title centered between horizontal rules
        %    - Four authors arranged in a 2 x 2 grid
        %    - Abstract placed near the lower part of the page
        %    - Keywords placed below the abstract
        %  Notes:
        %    - This layout assumes that commands such as \papertitle,
        %      \titleseparator, and all ORCID macros are defined elsewhere.
        %    - Hyperlinks are assumed to use the document's configured link color.
        % ======================================================================
        \begingroup\thispagestyle{empty}
        	\thispagestyle{empty} % Remove headers and footers from the title page
        	
        	\begin{center}
        		\medskip
% Flexible vertical space above the title block for visual balance
        		
        		% ------------------------------------------------------------------
        		%  TITLE BLOCK
        		%  - The title is visually framed by two separator rules.
        		%  - \titleseparator is assumed to be a custom rule macro defined
        		%    elsewhere in the document preamble.
        		% ------------------------------------------------------------------
        		\titleseparator % Upper title rule
        		 % Space between upper rule and paper title
        		
        		{\LARGE\bfseries \papertitle\par} % Paper title in large bold font
        		
        		 % Space between paper title and lower rule
        		\titleseparator % Lower title rule
        		
        		\medskip
% Flexible vertical space between title and author area
        		
        		% ------------------------------------------------------------------
        		%  AUTHOR BLOCKS
        		%  Layout:
        		%    - First row:  Author 1 | Author 2
        		%    - Second row: Author 3 | Author 4
        		%
        		%  Technical notes:
        		%    - tabular is used for horizontal placement.
        		%    - [t] top-aligns each row consistently.
        		%    - minipage[t] ensures each author block starts at the same top line.
        		%    -  inside minipage stabilizes top alignment.
        		%    - The whole block is suppressed by \blindtrue for a
        		%      double-blind submission.
        		% ------------------------------------------------------------------

        		\ifblind
        		{\large\textit{Author names and affiliations removed for double-blind review.}\par}
        		\else
        		% =========================
        		%  First row of authors
        		% =========================
        		\begin{tabular}[t]{cc}
        			
        			% --- Author 1: left column ---
        			\begin{minipage}[t]{0.44\textwidth}
        				 % Force true top alignment within the tabular cell
        				\centering % Center all content inside this author block

        				{\large\textbf{\authorKrumbiegel{}}~\orcidlink{\orcidKrumbiegel}\par} % Author name + ORCID
        				 % Space between author name and affiliation

        				{\small
        					Department of Mathematics, Saarland University,\\
        					Campus E1.1, 66123 Saarbrücken, Germany\par
        				} % Affiliation and postal address

        				 % Space between affiliation and contact line

        				{\small
        					\href{mailto:felix.krumbiegel@uni-saarland.de}{\texttt{felix.krumbiegel@uni-saarland.de}}\par
        				} % Email address
        			\end{minipage}
        			&
        			
        			% --- Author 2: right column ---
        			\begin{minipage}[t]{0.44\textwidth}
        				 % Force true top alignment within the tabular cell
        				\centering % Center all content inside this author block

        				{\large\textbf{\authorRogava{}}~\orcidlink{\orcidRogava}\par} % Author name + ORCID
        				 % Space between author name and affiliation

        				{\small
        					Ivane Javakhishvili Tbilisi State University, Ilia Vekua Institute of Applied Mathematics,\\
        					11 University Street, 0186, Tbilisi, Georgia\par
        				} % Affiliation and postal address split across lines for readability

        				 % Space between affiliation and contact line

        				{\small
        					\href{mailto:jemal.rogava@tsu.ge}{\texttt{jemal.rogava@tsu.ge}}\par
        				} % Email address
        			\end{minipage}
        			
        		\end{tabular}
        		
        		 % Fixed vertical separation between the two author rows
        		
        		% ==========================
        		%  Second row of authors
        		% ==========================
        		\begin{tabular}[t]{cc}
        			
        			% --- Author 3: left column ---
        			\begin{minipage}[t]{0.44\textwidth}
        				 % Force true top alignment within the tabular cell
        				\centering % Center all content inside this author block

        				{\large\textbf{\authorRupp{}}~\orcidlink{\orcidRupp}\par} % Author name + ORCID
        				 % Space between author name and affiliation

        				{\small
        					Department of Mathematics, Saarland University,\\
        					Campus E1.1, 66123 Saarbrücken, Germany\par
        				} % Affiliation and postal address

        				 % Space between affiliation and contact line

        				{\small
        					\href{https://rupp.ink}{\texttt{rupp.ink}} \textbar\ %
        					\href{mailto:andreas.rupp@uni-saarland.de}{\texttt{andreas.rupp@uni-saarland.de}}\par
        				} % Website and email address
        			\end{minipage}
        			&
        			
        			% --- Author 4: right column ---
        			\begin{minipage}[t]{0.44\textwidth}
        				 % Force true top alignment within the tabular cell
        				\centering % Center all content inside this author block

        				{\large\textbf{\authorVashakidze{}}~\orcidlink{\orcidVashakidze}\par} % Author name + ORCID
        				 % Space between author name and affiliation

        				{\small
        					School of Science and Technology, The University of Georgia (UG),\\
        					77 Merab Kostava Street, 0171, Tbilisi, Georgia;\\[0.4em]
        					Ilia Vekua Institute of Applied Mathematics,\\
        					11 University Street, 0186, Tbilisi, Georgia\par
        				} % Two affiliations, separated visually for clarity

        				 % Space between affiliation and contact line

        				{\small
        					\href{mailto:z.vashakidze@ug.edu.ge}{\texttt{z.vashakidze@ug.edu.ge}} \textbar\ %
        					\href{mailto:zurab.vashakidze@tsu.ge}{\texttt{zurab.vashakidze@tsu.ge}}\par
        				} % Two email addresses
        			\end{minipage}
        		\end{tabular}
        		
        		\fi % end \ifblind (author blocks)

        		\medskip
% Flexible vertical space to push the abstract toward the lower page area
        		
        		% ------------------------------------------------------------------
        		%  ABSTRACT + KEYWORDS BLOCK
        		%  - minipage narrows the text width for a more polished appearance.
        		%  - The abstract environment is assumed to be styled elsewhere.
        		% ------------------------------------------------------------------
        		\begin{minipage}{0.9\textwidth}
        			
        			\begin{abstract}
				\paperabstract
        				
        			\end{abstract}
        			% Abstract text

        			\medskip
% Flexible separation between abstract and keywords

        			% --- Keywords line ---
        			% - \noindent suppresses paragraph indentation.
        			% - \textbf emphasizes the keyword label.

        			\noindent\textbf{MSC 2020:} \papermsc.

        		\end{minipage}
        		
        	\end{center}
        \endgroup
		
		% ======================================================================
		%  METHOD: New Page
		%  Title: Hard Page Break
		%  Role: Start main content on a fresh page; stronger than \pagebreak.
		% ======================================================================
		\newpage

%% file: paper1-body.tex
		\section{Introduction}

        \subsection{Problem Formulation}\label{sec:problem_formulation}

        Kirchhoff-type equations constitute an important class of nonlocal partial differential equations arising from the classical model proposed by Kirchhoff as an extension and refinement of the linear theory of vibrating strings; see \cite{kirchhoff1876vorlesungen}. In this model, the tension is allowed to vary according to the total stretching induced by the displacement. In contrast to standard quasilinear hyperbolic equations, the principal part of a Kirchhoff problem depends on a global quantity involving the solution itself, typically its Dirichlet energy. This modification leads to a nonlinear and nonlocal evolution equation and serves as a prototype for the study of wave propagation in elastic media with displacement-dependent mechanical properties.
		
		Consider the nonlinear Kirchhoff-type integro-differential equation in two spatial dimensions
		\begin{equation}\label{eq:main_kirchhoff}
			\frac{\partial^2 u}{\partial t^2} - \Bigl( \alpha(t)  + \beta(t)  \dint{\left| \nabla u \right|^2} \Bigr) \Delta u = f\,,
		\end{equation}
		posed in $\left[ 0,T \right]$ and a bounded convex domain $\domain \subset \R^2$, which in particular has a Lipschitz boundary, together with the initial and boundary conditions
		\begin{equation}\label{eq:initial_conds}
			u ( 0 ) = u_0\,,\qquad \partial_t u ( 0 ) = v_0\,,\qquad u ( t ) |_{\partial \domain} = 0 \quad \text{for } t \in \left[ 0,T \right]\,.
		\end{equation}
		Here $u$ is the transverse displacement, $\left| \, \cdot\, \right|$ the Euclidean norm and $\nabla$ the gradient. Throughout both this paper and the companion paper~\cite{KirchhoffFEM}, $u$ and $f$ are regarded as functions of $t$ with values in a space of functions on $\domain$, so that the spatial argument is never displayed: $u ( t )$ is the state at time $t$, $u ( t_k )$ the exact solution at the $k$-th time level~$t_k$, and $u_k$ its approximation. The coefficients are scalar functions of $t$ and do carry that argument: $\alpha$ and $\beta$ are continuous and continuously differentiable on $\left[ 0,T \right]$, and we write $\alpha_{\min} = \min_{[0,T]} \alpha ( t )$, $\alpha_{\max} = \max_{[0,T]} \alpha ( t )$, $\beta_{\min} = \min_{[0,T]} \beta ( t )$ and $\beta_{\max} = \max_{[0,T]} \beta ( t )$, assuming $\alpha_{\min} > 0$ and $\beta_{\min} > 0$. The source term is taken in $f \in C \bigl( \left[ 0,T \right] ; L^2 ( \domain ) \bigr)$ throughout, so that the grid values $f_k = f ( t_k )$ that the schemes sample are defined and $\max_k \left\| f_k \right\|$ is bounded. Continuity, and not merely essential boundedness, is what makes the point values meaningful. Individual results impose more regularity on $f$ where they need it, in \eqref{eq:lemma2_data_regularity} and \thmref{prop:conver_thm_nonlinear}.

        The nonlocal character of \eqref{eq:main_kirchhoff} is its essential feature: unlike the classical wave equation, the coefficient multiplying the Laplacian is not prescribed independently of the solution but depends on a global quantity, the spatial deformation energy at time $t$, so that the evolution at each point is coupled to the behaviour of the solution throughout $\domain$. Mechanically, this is the variation of the effective tension induced by deformation: in two dimensions \eqref{eq:main_kirchhoff} models the transverse vibrations of an elastic membrane whose effective tension depends on its total deformation energy, with $\alpha $ the baseline tension and $\beta $ the scale of the nonlinear contribution. For time-independent parameters, the model arises in contemporary work on nanomechanical oscillators \cite{Yao2022}; allowing the parameters to depend on time reflects changes in the material properties during vibration, as caused at those scales by thermal fluctuations or strong external driving \cite{Samanta2023,Sun2006}, and so covers membranes subjected to varying temperature, pressure or humidity.

		\subsection{Related work and contribution}
		Numerical methods for Kirchhoff-type equations have a long history. Christie and Sanz-Serna \cite{Christie1984} reduced a one-dimensional homogeneous problem with constant coefficients to a first-order system and combined a finite element method in space with a Crank--Nicolson scheme in time. {In \cite{Peradze2005}, Peradze employed and analyzed the time-discretization scheme proposed by Christie and Sanz-Serna for nonlinear homogeneous Kirchhoff equations with constant coefficients, combined with a Galerkin discretization in the spatial variable, with sine functions chosen as both trial and test functions.} Symmetric three-layer semi-discrete schemes for abstract analogues of the Kirchhoff equation were introduced and analysed by Rogava and Tsiklauri \cite{RogTsikl2012,RogTsikl2014}, who established local convergence, and extended to Ball-type integro-differential equations in \cite{RogTsikVash2023}. Convergence for the nonlinear dynamic Kirchhoff string equation was obtained in \cite{VashGMJ2022}, {while a discretization based on Legendre polynomials was proposed in \cite{Vash2020}}. Closest to the present work is \cite{RogVashZAMM2024}, where a three-layer semi-discrete scheme for the one-dimensional string equation with \emph{time-dependent} coefficients is shown to converge.

		The present paper is a study of the \emph{temporal} discretization, and its
		contributions are the following. We consider two symmetric
		three-layer schemes, a locally linear and a genuinely nonlinear one, and show that each
		preserves a discrete analogue of the total mechanical energy of the homogeneous problem
		with constant coefficients, both identically rather
		than to within a discretization error. We further show that a third quantity, built from the potential energy and a centred difference approximation of
		the kinetic energy, converges with refinement of the temporal grid to the mechanical energy with second order. These
		identities are new for the two-dimensional problem and are what the word
		\emph{conservative} refers to throughout.

		For the nonlinear scheme, we obtain uniform bounds on the strain of the discrete displacement~$u_k$ and on the discrete velocity $(u_k - u_{k-1})/\tau$, where~$\tau$ denotes the time step size, without recourse to a discrete Gr\"onwall inequality. Further, we show local uniform bounds for higher spatial derivatives of the aforementioned quantities. The bounds are obtained by working
		directly with the discrete energies, which gives more explicit control of the quantities
		involved; however, we do not claim that this removes the exponential dependence on the final time,
		which the constants of \lemref{prop:lemma1_nonlinear} and \lemref{prop:lemma2_nonlinear}
		still carry. Furthermore, we establish the local second-order
		convergence in time both for the solution and for the central-difference
		approximation of its first time derivative.

		Finally, we analyse the
		nonlinear equation solved at each time step, and show that the associated iteration
		converges at a geometric rate once $\tau$ is small enough. This settles the one-step
		problem under the bounds established beforehand. Because the \apriori{} bounds of
		\Cref{sec:nonlinear_scheme} are index-local, alternating them with that solver constructs the
		trajectory stepwise, for time-dependent $\alpha$ and $\beta$ as well, on the local
		interval those bounds cover, see \Cref{rem:one_step_only}.

        {
        Furthermore, formulating the problem in terms of a self-adjoint, positive-definite extension of the negative Laplacian broadens the scope of the proposed time-stepping scheme. In particular, the analysis applies not only to conventional settings but also to problems posed in multiple spatial dimensions on Lipschitz domains. In addition, the discussion extends to various boundary conditions, provided that the negative Laplacian remains positive definite under the corresponding conditions. Thus, with the appropriate functional-analytic setting, the framework applies to Dirichlet, Robin, mixed, and, subject to the necessary restriction, Neumann boundary value problems.
        }

		The paper is organised as follows. \Cref{sec:semidiscrete_schemes} introduces the two three-layer schemes, \Cref{sec:energy_conservation} derives the
		energy conservation property of the homogeneous continuous problem with constant
		coefficients, and establishes the discrete energies for both
		schemes. \Cref{sec:nonlinear_scheme} contains the analysis of the nonlinear scheme:
		\apriori{} bounds, the error equation, and local second-order convergence in time.
		\Cref{sec:solvability} treats the nonlinear problem solved at each time level,
		proving existence, uniqueness, and geometric convergence of the associated iteration. \Cref{sec:numerics} reports numerical experiments and
		\Cref{sec:conclusion} concludes. The companion paper~\cite{KirchhoffFEM} extends the
		analysis of the present work to a fully discrete method.

		\section{Symmetric Three-Layer Semi-Discrete Schemes}\label{sec:semidiscrete_schemes}
		
		\noindent Consider a uniform partition of the temporal domain $\left[ 0,T \right]$ with mesh size $\tau = T / n$, where $n > 1$
		denotes the number of subintervals. The associated grid points are given by $t_k = k \tau$ for $k = 0,\ldots,n$. %
		An approximate solution to the problem \eqref{eq:main_kirchhoff}--\eqref{eq:initial_conds} is sought through the following semi-discrete scheme ({\cf} \cite{RogTsikl2012,RogVashZAMM2024})
		\begin{equation}\label{eq:linear_semidiscrete}
			\frac{\hat{\Delta}^2 u_{k - 1}}{\tau^2} - \tfrac{1}{2} q_k \left( \Delta u_{k + 1} + \Delta u_{k - 1} \right) = f_k\,,\quad k = 1,2,\ldots,n - 1\,,
		\end{equation}
		where $f_k = f ( t_k )$ is the evaluation at the grid points, $\hat{\Delta} u_{k - 1} = u_k - u_{k - 1}, \hat{\Delta}^2 u_{k - 1} = \hat{\Delta} ( \hat{\Delta} u_{k - 1} )$ are the first and second discrete differences, and the nonlinearity evaluated at the time step $t_k$ is given by 
		\begin{alignat*}{2}
			q_k &= \alpha_k + \beta_k \dint{\left| \nabla u_k \right|^2}\,, &\qquad \alpha_k &= \alpha( t_k )\,,\quad \beta_k = \beta( t_k )\,.
		\end{alignat*}
		Then, $u_k$ as an approximation of $u ( t_k )$ corresponding to the $k$-th time level. %
		To proceed with finding an approximate solution using the proposed scheme \eqref{eq:linear_semidiscrete}, one first specifies the initial functions $u_0$ and $u_1$, which are given as 
		\begin{equation}\label{eq:starting_vec1}
			u_0 = u( 0 )\,,\qquad u_1 = u_0 + \tau v_0 + \frac{\tau^2}{2} w_0\,,\qquad w_0 = f_0 + q_0 \Delta u_0\,.
		\end{equation}
		Using arguments similar to those developed in \cite{RogTsikl2014,RogVashZAMM2024}, one concludes that the scheme \eqref{eq:linear_semidiscrete} attains second-order accuracy in time; that is, it has an approximation order of $\bigO ( \tau^2 )$. This scheme may be referred to as a Crank–Nicolson-type locally linear scheme, since a linear problem is solved at each temporal layer.
		It should be noted that the following semi-discrete scheme likewise possesses an approximation order of $\bigO ( \tau^2 )$ and takes the form
		\begin{equation}\label{eq:nonlinear_semidiscrete}
			\frac{\hat{\Delta}^2 u_{k - 1}}{\tau^2} - \tfrac{1}{2}\Bigl( \alpha_k + \frac{\beta_k}{2} \dint{\bigl( \left| \nabla u_{k + 1} \right|^2 + \left| \nabla u_{k - 1} \right|^2 \bigr)} \Bigr) \left( \Delta u_{k + 1} + \Delta u_{k - 1} \right) = f_k\,.
		\end{equation}
		For the scheme given by \eqref{eq:nonlinear_semidiscrete}, the initial data are the same as in the preceding scheme \eqref{eq:linear_semidiscrete}. In contrast to \eqref{eq:linear_semidiscrete}, the scheme \eqref{eq:nonlinear_semidiscrete} may be referred to as a Crank–Nicolson-type nonlinear scheme.

        \section{Energy Conservation for the Homogeneous Kirchhoff Equation with Constant Coefficients}\label{sec:energy_conservation}
        
		\noindent Let us consider the homogeneous Kirchhoff equation \eqref{eq:main_kirchhoff} with constant coefficients $\alpha ,\beta \in \mathbb{R}_+$, subject to the same initial–boundary conditions \eqref{eq:initial_conds}: 
		\begin{equation}\label{eq:homog_kirchhoff}
			\frac{\partial^2 u}{\partial t^2} - \Bigl( \alpha + \beta \dint{\left| \nabla u \right|^2} \Bigr) \Delta u = 0\,.
		\end{equation}
		Multiplying \eqref{eq:homog_kirchhoff} by $\partial u / \partial t$, integrating over $\domain$, applying Green's formula and using $\left[ \partial u / \partial t \right]_{\partial \domain} = 0$, all under the assumption that the functions involved are sufficiently regular, gives
		\begin{equation}\label{eq:diff_energy}
			\frac{\dd \mathcal{K} }{\dd t} + \alpha \frac{\dd V }{\dd t} + \beta \frac{\dd V^2 }{\dd t} = 0\,,\quad \text{where}\quad
			\mathcal{K} ( t ) = \tfrac{1}{2} \dint{\Bigl( \frac{\partial u( t )}{\partial t} \Bigr)^2}\,,\quad V ( t ) = \tfrac{1}{2} \dint{\left| \nabla u ( t ) \right|^2}
		\end{equation}
		are the kinetic and the potential (strain) energy functionals. Integrating \eqref{eq:diff_energy} from $0$ to $t$ and setting
		\begin{equation}\label{eq:total_energy}
			E ( t ) = \mathcal{K} ( t ) + \alpha V ( t ) + \beta V^2 ( t )
		\end{equation}
		yields
		\begin{equation*}
			E ( t ) = E ( 0 ) = \tfrac{1}{2} \dint{v_0^2} + \frac{\alpha}{2} \dint{\left| \nabla u_0 \right|^2} + \frac{\beta}{4} \Bigl( \dint{\left| \nabla u_0 \right|^2} \Bigr)^2\,.
		\end{equation*}
		Energy conservation therefore remains valid for the present problem, although the potential energy now carries a nonlinear contribution.

		\subsection{An Analogue of the Energy Conservation Property for Semi-Discrete Numerical Schemes}\label{subsec:energy_discrete}

		Let us consider the scheme \eqref{eq:linear_semidiscrete} for equation \eqref{eq:homog_kirchhoff}:
		\begin{equation}\label{eq:linear_semidiscrete_homog}
			\frac{\hat{\Delta}^2 u_{k - 1}}{\tau^2} - \tfrac{1}{2} q_k \left( \Delta u_{k + 1} + \Delta u_{k - 1} \right) = 0\,.
		\end{equation}
		Multiplying \eqref{eq:linear_semidiscrete_homog} by $u_{k + 1} - u_{k - 1} = \hat{\Delta} u_k + \hat{\Delta} u_{k - 1}$, integrating over $\domain$, expanding the brackets in the first integral, and applying Green's formula to the second gives, with the discrete counterparts
		\begin{equation}\label{eq:notation_discrete}
			\mathcal{K}_k = \tfrac{1}{2} \dint{\Bigl( \frac{\hat{\Delta} u_{k - 1}}{\tau} \Bigr)^2}\,,\quad V_k = \tfrac{1}{2} \dint{\left| \nabla u_k \right|^2}\,,
		\end{equation}
		the relation $2 \mathcal{K}_{k + 1} + ( \alpha + 2 \beta V_k ) V_{k + 1} = 2 \mathcal{K}_{k} + ( \alpha + 2 \beta V_k ) V_{k - 1}$. Writing the right-hand side as $2 \mathcal{K}_{k} + ( \alpha + 2 \beta V_{k - 1} ) V_k + \alpha ( V_{k - 1} - V_k )$ makes the first two terms the same expression one level down, and the residue $\alpha ( V_{k-1} - V_k )$ telescopes. Summing from $1$ to $k$ and dividing by 2 therefore leaves
		\begin{equation*}
			\mathcal{K}_{k + 1} + \tfrac{\alpha}{2} \left( V_k + V_{k + 1} \right) + \beta V_k V_{k + 1} = \mathcal{K}_{1} + \tfrac{\alpha}{2} ( V_0 + V_1 ) + \beta V_0 V_1\,.
		\end{equation*}
		With $\widetilde{E}_{1,k} = \mathcal{K}_{k} + \tfrac{\alpha}{2} \left( V_{k - 1} + V_{k} \right) + \beta V_{k - 1} V_k$ this reads $\widetilde{E}_{1,k + 1} = \widetilde{E}_{1,1}$ for $k = 1,2,\ldots,n - 1$. The locally linear symmetric three-layer semi-discrete scheme is energy-preserving (conservative) in a discrete sense. 

		Applying the scheme \eqref{eq:nonlinear_semidiscrete} to equation \eqref{eq:homog_kirchhoff} yields
		\begin{align}\label{eq:nonlinear_semidiscrete_homog}
			&\frac{\hat{\Delta}^2 u_{k - 1}}{\tau^2} - \tfrac{1}{2}\Bigl( \alpha + \tfrac{\beta}{2} \dint{\left( \left| \nabla u_{k + 1} \right|^2 + \left| \nabla u_{k - 1} \right|^2 \right)} \Bigr) \left( \Delta u_{k + 1} + \Delta u_{k - 1} \right) = 0\,.
		\end{align}
		The same test function and the same manipulation, now with $\alpha V_k / 2$ and $\beta V_k^2 / 2$ added to both sides, give
		\begin{equation*}
			\mathcal{K}_{k + 1} + \tfrac{\alpha}{2} \left( V_k + V_{k + 1} \right) + \tfrac{\beta}{2} \left( V_k^2 + V_{k + 1}^2 \right) = \mathcal{K}_k + \tfrac{\alpha}{2} \left( V_{k - 1} + V_k \right) + \tfrac{\beta}{2} \left( V_{k - 1}^2 + V_k^2 \right)\,,
		\end{equation*}
		so that, with $\widetilde{E}_{2,k} = \mathcal{K}_k + \tfrac{\alpha}{2} \left( V_{k - 1} + V_k \right) + \tfrac{\beta}{2} \left( V_{k - 1}^2 + V_k^2 \right)$, we obtain $\widetilde{E}_{2,k + 1} = \widetilde{E}_{2,k} = \widetilde{E}_{2,1}$ for $k = 1,2,\ldots,n - 1$. The quantity $\widetilde{E}_{2,k}$ is thus an energy invariant of the scheme \eqref{eq:nonlinear_semidiscrete_homog}, replicating the conservation property of the continuous problem \eqref{eq:homog_kirchhoff}.

		\section{Investigation of the Crank–Nicolson-type Nonlinear Scheme}\label{sec:nonlinear_scheme}
		
		The lemmas presented in this section play a principal role in establishing the main convergence theorem as well as the convergence of the iteration process associated with the Crank–Nicolson-type nonlinear scheme. %
		Let us define the notation:
		\begin{equation}\label{eq:laplacian}
			\L_0 = - \Delta\,,\quad D ( \L_0 ) = \left\{ u \in C^2 \left( \overline{\domain} \right) : \left.u\right|_{\partial\domain} = 0 \right\}\,.
		\end{equation}
		Leveraging the notation introduced in \eqref{eq:laplacian} and invoking Green's formula, we may rewrite the nonlinear scheme \eqref{eq:nonlinear_semidiscrete} in the form
		\begin{equation}\label{eq:nonlinear_scheme_laplacian}
			\frac{\hat{\Delta}^2 u_{k - 1}}{\tau^2} + \tfrac{1}{2}\Bigl( \alpha_k + \tfrac{\beta_k}{2} \dint{\left( \L_0 u_{k + 1} u_{k + 1} + \L_0 u_{k - 1} u_{k - 1} \right)} \Bigr) \left( \L_0 u_{k + 1} + \L_0 u_{k - 1} \right) = f_k\,.
		\end{equation}
		
		In what follows, we denote by $\left( \cdot,\cdot \right)$ the standard inner product in $L^2 ( \domain )$, and by $\left\| \,\cdot\, \right\|$ the corresponding $L^2$-norm. %
		It is well known that the operator $\L_0$ is symmetric and positive definite. In particular, the following inequality is satisfied (see Chap.~18 in \cite{Rektorys1980})
		\begin{equation}\label{eq:pos_def}
			( \L_0 u,u ) \geq \tfrac{2 \pi^2}{\ell_\domain^2} \left\| u \right\|^2\,,\quad u \in D ( \L_0 )\,.
		\end{equation}
		In this context, $\ell_\domain$ denotes the length of the side of the square enclosing the domain $\domain$.
		
		By virtue of the Friedrichs extension theorem, any densely defined, symmetric, positive definite operator admits a self-adjoint, positive definite extension ({\cf} \cite{Maurin1972}). Throughout the remainder of the paper, $\L$ denotes the Friedrichs extension of $\L_0$, the unique self-adjoint, positive definite operator associated with the Dirichlet form $a(u,v) = ( \nabla u, \nabla v )$ on $H_0^1 ( \domain )$, so that $\L_0 \subset \L$ and $D \bigl( \L^{1/2} \bigr) = H_0^1 ( \domain )$. Since $\L$ is self-adjoint and positive definite, its fractional powers $\L^{a}$, $a > 0$, are defined by the spectral calculus, with
		\begin{equation*}
			D \bigl( \L^{a} \bigr) = \Bigl\{ u \in L^2 ( \domain ) : \int_0^\infty \lambda^{2a} \, \dd ( E_\lambda u, u ) < \infty \Bigr\}\,,\quad \bigl\| \L^{a} u \bigr\|^2 = \int_0^\infty \lambda^{2a} \, \dd ( E_\lambda u, u )\,,
		\end{equation*}
		where $\left\{ E_\lambda \right\}$ is the spectral resolution of $\L$. It should be noted that inequality \eqref{eq:pos_def} remains valid for $\L$ on $D ( \L )$.
        Employing the operator $\L$, equation \eqref{eq:nonlinear_scheme_laplacian} takes the following form:
		\begin{equation}\label{eq:nonlinear_scheme_extended}
			\frac{\hat{\Delta}^2 u_{k - 1}}{\tau^2} + \tfrac{1}{2} \Bigl( \alpha_k + \tfrac{\beta_k}{2} \bigl( \| \L^{1/2} u_{k + 1} \|^2 + \| \L^{1/2} u_{k - 1} \|^2 \bigr) \Bigr) \left( \L u_{k + 1} + \L u_{k - 1} \right) = f_k\,,
		\end{equation}
		where $k = 1,2,\ldots,n - 1$.
		
		\subsection{Principal Lemmas}
		In the following lemma, we establish uniform $L^2$-bounds for the sequences of functions $\hat{\Delta} u_{k - 1} / \tau$ and $\L^{1/2} u_k$. This property is necessary both for proving the convergence of the approximate solution and for establishing the convergence of the solution-finding iteration process. It is important to note that this issue is nontrivial, since the energy method does not provide a recurrence inequality suitable for applying the telescoping series cancellation technique. To obtain the desired results, one would typically rely on a discrete Gr\"onwall-type inequality; however, in our case, we were able to derive the final estimates without invoking such an inequality. 
		
		\begin{lemma}\label{prop:lemma1_nonlinear}
            Assume that the initial values in \eqref{eq:starting_vec1} satisfy $u_0, v_0, w_0 \in D ( \L^{1/2} )$. Then the initial energy $E_1$ admits a bound that is uniform in $\tau$, although $E_1$ itself depends on $\tau$ through $u_1$. Then $\hat{\Delta} u_{k - 1} / \tau$ and $\L^{1/2} u_k$ are uniformly bounded in $L^2 ( \domain )$, \ie, there are positive constants $\cstmone$ and $\cstmtwo$, depending only on these data, on $\alpha$, $\beta$, $T$ and $\max_{0 \leq t \leq T} \| f ( t ) \|$, and hence independent of $\tau$, such that the inequalities
			\begin{equation}\label{eq:unif_bound_half_lu}
				\| \hat{\Delta} u_{k - 1} /\tau \| \leq \cstmone\quad \text{and}\quad \| \L^{1/2} u_k \| \leq \cstmtwo
			\end{equation}
			hold for every $k = 1,2,\ldots,n$. 
		\end{lemma}
		\begin{proof}
			By taking the inner product of equation \eqref{eq:nonlinear_scheme_extended} with $u_{k + 1} - u_{k - 1} = \hat{\Delta} u_k + \hat{\Delta} u_{k - 1}$ for $k = 1,2,\ldots,n - 1$ and using the fact that the operator $\L$ is self-adjoint and positive definite, we arrive, with $\mathcal{K}_k$ and $V_k$ as in~\eqref{eq:notation_discrete}, at
			\begin{equation}\label{eq:lemma1_nonlinear_first_eqt}
				\mathcal{K}_{k + 1} + \tfrac{\alpha_k}{2} V_{k + 1} + \tfrac{\beta_k}{2} V_{k + 1}^2 = \mathcal{K}_k + \tfrac{\alpha_k}{2} V_{k - 1} + \tfrac{\beta_k}{2} V_{k - 1}^2 + \tfrac{1}{2} ( f_k, \hat{\Delta} u_k ) + \tfrac{1}{2} ( f_k, \hat{\Delta} u_{k - 1} )\,,
			\end{equation}
			Adding the terms $\alpha_k V_k / 2$ and $\beta_k V_k^2 / 2$ to both sides of the relation \eqref{eq:lemma1_nonlinear_first_eqt} yields
			\begin{align*}
				&\mathcal{K}_{k + 1} + \tfrac{\alpha_k}{2} ( V_k + V_{k + 1} ) + \tfrac{\beta_k}{2} ( V_k^2 + V_{k + 1}^2 ) \\
				&\qquad = \mathcal{K}_k + \tfrac{\alpha_k}{2} ( V_{k - 1} + V_k ) + \tfrac{\beta_k}{2} ( V_{k - 1}^2 + V_k^2 ) + \tfrac{1}{2} ( f_k, \hat{\Delta} u_k ) + \tfrac{1}{2} ( f_k, \hat{\Delta} u_{k - 1} )\,.
			\end{align*}
			By introducing the notation $E_k = \mathcal{K}_k + \tfrac{\alpha_{k - 1}}{2} ( V_{k - 1} + V_k ) + \tfrac{\beta_{k - 1}}{2} ( V_{k - 1}^2 + V_k^2 )$, the preceding equality may be rewritten in the form
			\begin{equation}\label{eq:lemma1_nonlinear_transf_eq}
				E_{k + 1} = E_k + \tfrac{1}{2} ( \alpha_k - \alpha_{k - 1} ) ( V_{k - 1} + V_k ) + \tfrac{1}{2} ( \beta_k - \beta_{k - 1} ) ( V_{k - 1}^2 + V_k^2 ) + \tfrac{1}{2} ( f_k, \hat{\Delta} u_k ) + \tfrac{1}{2} ( f_k, \hat{\Delta} u_{k - 1} )\,.
			\end{equation}
			It is obvious that the coefficient functions $\alpha ( t )$ and $\beta ( t )$ satisfy the following inequalities:
			\begin{align}
				\left| \alpha_k - \alpha_{k - 1} \right| &\leq \lmonecstalph \tau \quad \text{with}\quad \lmonecstalph = \max_{0 \leq t \leq T} \left| \alpha^\prime ( t ) \right|\,,\label{eq:lemma1_nonlinear_alpha_est} \\
				\left| \beta_k - \beta_{k - 1} \right| &\leq \lmonecstbeta \tau \quad \text{with}\quad \lmonecstbeta = \max_{0 \leq t \leq T} \left| \beta^\prime ( t ) \right|\,.\label{eq:lemma1_nonlinear_beta_est}
			\end{align}
			Invoking inequalities \eqref{eq:lemma1_nonlinear_alpha_est} and \eqref{eq:lemma1_nonlinear_beta_est}, along with the conditions $\alpha ( t ) \geq \alpha_{\min} > 0$ and $\beta ( t ) \geq \beta_{\min} > 0$, we derive an estimate for the following term:
			\begin{align}\label{eq:lemma1_nonlinear_term2_3}
				&\tfrac{1}{2} \left( \alpha_k - \alpha_{k - 1} \right) \left( V_{k - 1} + V_k \right) + \tfrac{1}{2} \left( \beta_k - \beta_{k - 1} \right) \left( V_{k - 1}^2 + V_k^2 \right)\nonumber \\
				&\qquad \leq \max\{ \lmonecstalph, \lmonecstbeta \} \tau \left[ \tfrac{1}{2} \left( V_{k - 1} + V_k \right) + \tfrac{1}{2} \left( V_{k - 1}^2 + V_k^2 \right) \right] \leq \lmonecstdiff \tau E_k\,,
			\end{align}
			the last step by inserting the factors $\alpha_{k-1}/\alpha_{k-1}$ and $\beta_{k-1}/\beta_{k-1}$, bounding the denominators below by $\alpha_{\min}$ and $\beta_{\min}$, and recognising $E_k$, thus, $\lmonecstdiff = \max \{ \lmonecstalph, \lmonecstbeta \} / \min \{ \alpha_{\min}, \beta_{\min} \}$.
			By applying the Cauchy–Schwarz inequality to the inner products below, we conclude that
			\begin{multline}\label{eq:lemma1_nonlinear_source_term}
				\tfrac{1}{2} ( f_k, \hat{\Delta} u_k ) + \tfrac{1}{2} ( f_k, \hat{\Delta} u_{k - 1} ) \leq \tfrac{1}{2} \| f_k \| ( \| \hat{\Delta} u_k \| + \| \hat{\Delta} u_{k - 1} \| ) \\
				= \tau\tfrac{1}{\sqrt{2}} ( \sqrt{\mathcal{K}_{k + 1}} + \sqrt{\mathcal{K}_k} ) \left\| f_k \right\| \leq \tau\tfrac{1}{\sqrt{2}} ( \sqrt{E_{k + 1}} + \sqrt{E_k} ) \left\| f_k \right\|\,.
			\end{multline}
			Substituting inequalities \eqref{eq:lemma1_nonlinear_term2_3} and \eqref{eq:lemma1_nonlinear_source_term} into relation \eqref{eq:lemma1_nonlinear_transf_eq} leads to the following conclusion:
			\begin{equation*}
				E_{k + 1} \leq ( 1 + \lmonecstdiff \tau ) E_k + \tau\tfrac{1}{\sqrt{2}} ( \sqrt{E_{k + 1}} + \sqrt{E_k} ) \left\| f_k \right\|\,.
			\end{equation*}
			Iterating this inequality and using $( 1 + \lmonecstdiff \tau )^k \leq e^{\lmonecstdiff t_k}$ gives
			\begin{equation}\label{eq:lemma1_nonlinear_recurrence_ineq}
				E_{k + 1} \leq e^{\lmonecstdiff t_k} \Bigl( E_1 + \tau\tfrac{1}{\sqrt{2}} \sum_{i = 1}^{k} \left( \sqrt{E_{i + 1}} + \sqrt{E_i} \right) \left\| f_i \right\| \Bigr)\,.
			\end{equation}
			We now apply the maximal-index device of \cite{RogTsikl2014,RogVashZAMM2024}, which is used
			twice more below. Let $E_j = \max_{1 \leq i \leq k + 1} E_i$. Then
			\eqref{eq:lemma1_nonlinear_recurrence_ineq} holds with $k+1$ replaced by $j$, and dividing
			it by $\sqrt{E_j}$ makes every ratio $\sqrt{E_i}/\sqrt{E_j}$ at most $1$, so that
			\begin{equation*}
				\sqrt{E_j} \leq e^{\lmonecstdiff t_k} \Bigl( \sqrt{E_1} + \sqrt{2} \tau \sum_{i = 1}^{j - 1} \left\| f_i \right\| \Bigr) \leq e^{\lmonecstdiff t_k} \bigl( \sqrt{E_1} + \sqrt{2} t_k \max_{1 \leq i \leq k}\left\| f_i \right\| \bigr)\,.
			\end{equation*}
			Since $E_{k + 1} \leq E_j$, it can be concluded that
			\begin{equation*}
				E_{k + 1} \leq e^{2 \lmonecstdiff T} {\Bigl( \sqrt{E_1} + \sqrt{2} T \max_{0 \leq t \leq T} \left\| f ( t ) \right\| \Bigr)}^2\,,\quad k = 1,2,\ldots,n - 1\,,
			\end{equation*}
			where $\max_{1 \leq i \leq n} \| f_i \| \leq \max_{0 \leq t \leq T} \| f ( t ) \|$, a quantity attached to $f$ and not to the grid, since $f \in C ( [0,T] ; L^2 ( \domain ) )$.
			For $E_1$, \eqref{eq:starting_vec1} gives $( u_1 - u_0 ) / \tau = v_0 + \tfrac{\tau}{2} w_0$ and $u_1 = u_0 + \tau v_0 + \tfrac{\tau^2}{2} w_0$, whence $\mathcal{K}_1 \leq ( \| v_0 \| + T \| w_0 \| )^2$ and $V_0, V_1 \leq ( \| \L^{1/2} u_0 \| + T \| \L^{1/2} v_0 \| + T^2 \| \L^{1/2} w_0 \| )^2$ for $0 < \tau \leq T$ free of $\tau$. We thus conclude that the sequences $\left( u_k - u_{k - 1} \right) / \tau$ and $\L^{1/2} u_k$ are uniformly bounded. %
		In particular, the discrete kinetic and potential energies are uniformly bounded.
		\end{proof}

		Convergence of the approximate solution and of the iteration process below depends further
		on the uniform boundedness of $\L^{1/2} ( \hat{\Delta} u_{k - 1} / \tau )$ and $\L u_k$. Its
		proof rests on the following nonlinear inequality, which we quote to keep the paper
		self-contained (see Lemma 3.2 in \cite{RogTsikl2012})
		\begin{lemma}[{\cf} Rogava and Tsiklauri \cite{RogTsikl2012}]\label{prop:lemma_RogTsikl2012}
			Let $\left\{ a_k \right\}_{k = 0}^{n}$ and $\left\{ b_k \right\}_{k = 0}^{n}$ be sequences of nonnegative real numbers satisfying
			\begin{equation*}
				a_{k + 1} \leq a_k ( 1 + \tau a_k^s ) + \tau b_k
			\end{equation*}
			for all $k = 0,1,\ldots,n - 1$, where $s$ and $\tau$ are strictly positive constants. %
			Then the following estimate follows
			\begin{equation*}
				a_k \leq \frac{a}{( 1 - s a^s c_k t_k )^{1/s}}\,,\quad t_k = k \tau < \frac{1}{ s a^s c_k}\,,\quad a = \max\{ 1, a_0 \}\,,\quad c_k = 1 + \max_{0 \leq i \leq k} b_i\,.
			\end{equation*}
		\end{lemma}
		
		\begin{lemma}\label{prop:lemma2_nonlinear}
			Assume that the data satisfy the operator-domain regularity
			\begin{equation}\label{eq:lemma2_data_regularity}
				u_0 \in D ( \L^2 )\,,\qquad
				v_0 \in D ( \L )\,,\qquad
				f \in C \bigl( \left[ 0,T \right] ; D \bigl( \L^{1/2} \bigr) \bigr)
				\quad\text{with}\quad f_0 \in D ( \L )\,,
			\end{equation}
			\noindent so that, by \eqref{eq:starting_vec1}, also
			$w_0 = f_0 - q_0 \L u_0 \in D ( \L )$.
			The sequences of functions $\L^{1/2} \left( u_k - u_{k - 1} \right) / \tau$ and $\L u_k$ are locally uniformly bounded in the $L^2 ( \domain )$. More precisely, there exists $\overline{T} \in ( 0,T ]$ such that
			\begin{equation}\label{eq:unif_bound_lu}
				\| \L^{1/2} (u_k - u_{k - 1})/\tau \| \leq \cstmthree\,,\quad \left\| \L u_k \right\| \leq \cstmfour\,,\quad 1 \leq k \leq \Bigl[ \frac{\overline{T}}{\tau} \Bigr]\,,
			\end{equation}
			where $\cstmthree$ and $\cstmfour$ are positive constants depending on the value of $\overline{T}$.
		\end{lemma}
		\begin{proof}
			For each $k = 1,2,\ldots,n - 1$, we take the inner product of equation \eqref{eq:nonlinear_scheme_extended} with $\L \left( u_{k + 1} - u_{k - 1} \right) = \L \left( \hat{\Delta} u_k \right) + \L \left( \hat{\Delta} u_{k - 1} \right)$. By virtue of the self-adjointness and positive definiteness of the operator $\L$, and using that $f_k \in D ( \L^{1/2} )$ for every $k$ by~\eqref{eq:lemma2_data_regularity}, we obtain the following equality:
			\begin{multline}\label{eq:lemma2_nonlinear_init_eqt}
				\hat{\mathcal{K}}_{k + 1} + \tfrac{1}{2} [ \alpha_k + \beta_k ( V_{k + 1} + V_{k - 1} ) ] \hat{V}_{k + 1} \\
				= \hat{\mathcal{K}}_k + \tfrac{1}{2} [ \alpha_k + \beta_k ( V_{k + 1} + V_{k - 1} ) ] \hat{V}_{k - 1} + \tfrac{1}{2} ( \L^{1/2} f_k, \L^{1/2} ( \hat{\Delta} u_k ) ) + \tfrac{1}{2} ( \L^{1/2} f_k, \L^{1/2} ( \hat{\Delta} u_{k - 1} )
			\end{multline}
			with
			\begin{equation*}
				\hat{\mathcal{K}}_k = \tfrac{1}{2} \bigl\| \L^{1/2} \frac{\hat{\Delta} u_{k - 1}}{\tau} \bigr\|^2\,,\quad\hat{V}_k = \tfrac{1}{2} \| \L u_k \|^2\,,\quad \text{and} \quad V_k = \tfrac{1}{2} \| \L^{1/2} u_k \|^2\,.
			\end{equation*}
			
			An application of the Cauchy–Schwarz inequality yields:
			\begin{multline}\label{eq:lemma2_nonlinear_Cauchy_Schwarz_ineq}
				( \L^{1/2} f_k, \L^{1/2} ( \hat{\Delta} u_k ) ) + ( \L^{1/2} f_k, \L^{1/2} ( \hat{\Delta} u_{k - 1} ) ) \\
				\leq \tau \| \L^{1/2} f_k \| \Bigl( \Bigl\| \L^{1/2} \frac{\hat{\Delta} u_k}{\tau} \Bigr\| + \Bigl\| \L^{1/2} \frac{\hat{\Delta} u_{k - 1}}{\tau} \Bigr\| \Bigr) = \tau \sqrt{2} \| \L^{1/2} f_k \| \bigl( \sqrt{\hat{\mathcal{K}}_{k + 1}} + \sqrt{\hat{\mathcal{K}}_k} \bigr)\,.
			\end{multline}
			Substituting inequality \eqref{eq:lemma2_nonlinear_Cauchy_Schwarz_ineq} into \eqref{eq:lemma2_nonlinear_init_eqt}, we obtain:
			\begin{multline}\label{eq:lemma2_nonlinear_main_ineq}
				\hat{\mathcal{K}}_{k + 1} + \tfrac{1}{2} \left[ \alpha_k + \beta_k \left( V_{k + 1} + V_{k - 1} \right) \right] \hat{V}_{k + 1} \\
				\leq \hat{\mathcal{K}}_k + \tfrac{1}{2} \left[ \alpha_k + \beta_k \left( V_{k + 1} + V_{k - 1} \right) \right] \hat{V}_{k - 1} + \tau \tfrac{1}{\sqrt{2}} \| \L^{1/2} f_k \| \Bigl( \sqrt{\hat{\mathcal{K}}_{k + 1}} + \sqrt{\hat{\mathcal{K}}_k} \Bigr)\,.
			\end{multline}
			Proceeding further, adding the term $\left[ \alpha_k + \beta_k \left( V_{k + 1} + V_{k - 1} \right) \right] \hat{V}_k / 2$ to both sides of inequality \eqref{eq:lemma2_nonlinear_main_ineq}, and introducing the notation $W_k = V_{k + 1} + V_{k - 1} \quad \text{and} \quad \hat{W}_k = \tfrac12 (\hat{V}_k + \hat{V}_{k - 1})$ leads to the following inequality
			\begin{equation}\label{eq:lemma2_nonlinear_before_recurrence}
				\hat{\mathcal{K}}_{k + 1} + ( \alpha_k + \beta_k W_k ) \hat{W}_{k + 1} \leq \hat{\mathcal{K}}_k + ( \alpha_k + \beta_k W_k ) \hat{W}_k + \tau \tfrac{1}{\sqrt{2}} \| \L^{1/2} f_k \| \Bigl( \sqrt{\hat{\mathcal{K}}_{k + 1}} + \sqrt{\hat{\mathcal{K}}_k} \Bigr)\,.
			\end{equation}
			To obtain the desired recurrence inequality, we perform an additional elementary manipulation. In particular, we add and subtract the terms $\alpha_{k - 1} \hat{W}_k$ and $\beta_{k - 1} W_{k - 1} \hat{W}_{k}$ on the right-hand side of inequality \eqref{eq:lemma2_nonlinear_before_recurrence}, and subsequently collect terms by factoring out the common factors. This yields:
			\begin{equation}\label{eq:lemma2_nonlinear_Tk_enq}
				\begin{aligned}
					&\hat{\mathcal{K}}_{k + 1} + ( \alpha_k + \beta_k W_k ) \hat{W}_{k + 1} \leq \hat{\mathcal{K}}_k + \left( \alpha_{k - 1} + \beta_{k - 1} W_{k - 1} \right) \hat{W}_k + \left( \alpha_k - \alpha_{k - 1} \right) \hat{W}_k \\
					&\qquad + \left( \beta_k W_k - \beta_{k - 1} W_{k - 1} \right) \hat{W}_k + \tau \tfrac{1}{\sqrt{2}} \| \L^{1/2} f_k \| \Bigl( \sqrt{\hat{\mathcal{K}}_{k + 1}} + \sqrt{\hat{\mathcal{K}}_k} \Bigr)\quad \text{for}\quad k = 1,2,\ldots,n - 1\,.
				\end{aligned}
			\end{equation}
			Under the assumption $u_{-1} = u_0$, it follows that $\mathcal{K}_0 = \hat{\mathcal{K}}_0 = 0$, $V_{-1} = V_0$, and $\hat{V}_{-1} = \hat{V}_0$. Moreover, applying \lemref{prop:lemma1_nonlinear} to estimate the absolute value of the difference $W_k - W_{k - 1}$ for $k = 1,2,\ldots,n - 1$, we find that:
			\begin{align}\label{eq:lemma2_nonlinear_diff_W}
				\left| W_k - W_{k - 1} \right| &= \left| \left( V_{k + 1} - V_k \right) + \left( V_{k - 1} - V_{k - 2} \right) \right|\nonumber \\
				&\leq \bigl( \sqrt{V_{k + 1}} + \sqrt{V_k} \bigr) \bigl| \sqrt{V_{k + 1}} - \sqrt{V_k} \bigr| + \bigl( \sqrt{V_{k - 1}} + \sqrt{V_{k - 2}} \bigr) \bigl| \sqrt{V_{k - 1}} - \sqrt{V_{k - 2}} \bigr|\nonumber \\
				&\leq \sqrt{2} \cstmtwo \tau \Bigl( \sqrt{\hat{\mathcal{K}}_{k + 1}} + \sqrt{\hat{\mathcal{K}}_{k - 1}} \Bigr)\,.
			\end{align}
			Combining inequalities \eqref{eq:lemma1_nonlinear_beta_est} and \eqref{eq:lemma2_nonlinear_diff_W} with \lemref{prop:lemma1_nonlinear}, one can obtain the following estimates:
			\begin{align}\label{eq:lemma2_nonlinear_diff_beta_w}
				\left| \beta_k W_k - \beta_{k - 1} W_{k - 1} \right| &= \left| \left( \beta_k - \beta_{k - 1} \right) W_k + \beta_{k - 1} \left( W_k - W_{k - 1} \right) \right|\nonumber \\
				&\leq \lmonecstbeta \cstmtwo^2 \tau + \sqrt{2} \cstmtwo \beta_{k - 1} \tau \Bigl( \sqrt{\hat{\mathcal{K}}_{k + 1}} + \sqrt{\hat{\mathcal{K}}_{k - 1}} \Bigr)\nonumber \\
				&\leq \lmtwocstdiffbw \tau \Bigl( 1 + \sqrt{\hat{\mathcal{K}}_{k + 1}} + \sqrt{\hat{\mathcal{K}}_{k - 1}} \Bigr)\,,\quad \lmtwocstdiffbw = \cstmtwo \max \Bigl\{ \lmonecstbeta \cstmtwo, \sqrt{2} \beta_{\max} \Bigr\}\,.
			\end{align}
			Employing inequalities \eqref{eq:lemma1_nonlinear_alpha_est} and \eqref{eq:lemma2_nonlinear_diff_beta_w}, and introducing the notation
			\begin{equation*}
				\hat{E}_k = \hat{\mathcal{K}}_k + \left( \alpha_{k - 1} + \beta_{k - 1} W_{k - 1} \right) \hat{W}_k\,,\quad k = 1,2,\ldots,n\,,
			\end{equation*}
			we may express inequality \eqref{eq:lemma2_nonlinear_Tk_enq}, with $\lmtwocsthatenerg = \lmonecstalph + \lmtwocstdiffbw$, in the following form:
			\begin{equation*}
				\hat{E}_{k + 1} \leq \hat{E}_k + \lmtwocsthatenerg \tau \hat{W}_k + \lmtwocstdiffbw \tau \Bigl( \sqrt{\hat{\mathcal{K}}_{k + 1}} + \sqrt{\hat{\mathcal{K}}_{k - 1}} \Bigr) \hat{W}_k + \tau \tfrac{1}{\sqrt{2}} \| \L^{1/2} f_k \| \Bigl( \sqrt{\hat{\mathcal{K}}_{k + 1}} + \sqrt{\hat{\mathcal{K}}_k} \Bigr)\,.
			\end{equation*}
			It may be observed that the following relations hold:
			\begin{equation*}
				\hat{\mathcal{K}}_k \leq \hat{E}_k \quad \text{and} \quad \hat{W}_k \leq \tfrac{1}{\alpha_{\min}} \hat{E}_k\,,\quad \text{for} \quad k = 1,2,\ldots,n\,.
			\end{equation*}
			In view of the relations above, the preceding inequality implies that
			\begin{gather*}
				\hat{E}_{k + 1} \leq ( 1 + \lmtwocstfracone \tau ) \hat{E}_k + \lmtwocstfractwo \tau \Bigl( \sqrt{\hat{E}_{k + 1}} + \sqrt{\hat{E}_{k - 1}} \Bigr) \hat{E}_k + \tfrac{\lmtwocstmaxf}{\sqrt{2}} \tau \Bigl( \sqrt{\hat{E}_{k + 1}} + \sqrt{\hat{E}_k} \Bigr)\,, \\
				\lmtwocstfracone = \tfrac{\lmtwocsthatenerg}{\alpha_{\min}}\,,\qquad \lmtwocstfractwo = \tfrac{\lmtwocstdiffbw}{\alpha_{\min}}\,,\qquad \lmtwocstmaxf = \max_{0 \leq t \leq T} \| \L^{1/2} f( t ) \|\,,
			\end{gather*}
			where we additionally assume that $\hat{E}_0 = \hat{\mathcal{K}}_0 = 0$.
			We introduce the notation $\overline{E}_k = \max_{1 \leq i \leq k} \hat{E}_i \quad \text{with} \quad \overline{E}_0 = \hat{E}_0 = 0$.
			Suppose that $\hat{E}_{i + 1}$, for $0 \leq i \leq k$, attains its maximum at $i = j$. Then, by the notation established above, we have:
			\begin{equation*}
				\overline{E}_{k + 1} = \max_{0 \leq i \leq k} \hat{E}_{i + 1} = \hat{E}_{j + 1}\,,\quad k = 1,2,\ldots,n - 1\,.
			\end{equation*}
			If $j \geq 1$, then it follows that:
			\begin{align*}
				\overline{E}_{k + 1} = \hat{E}_{j + 1} &\leq ( 1 + \lmtwocstfracone \tau ) \hat{E}_j + \lmtwocstfractwo \tau \Bigl( \sqrt{\hat{E}_{j + 1}} + \sqrt{\hat{E}_{j - 1}} \Bigr) \hat{E}_j + \tfrac{\lmtwocstmaxf}{\sqrt{2}} \tau \Bigl( \sqrt{\hat{E}_{j + 1}} + \sqrt{\hat{E}_j} \Bigr) \\
				&\leq ( 1 + \lmtwocstfracone \tau ) \overline{E}_k + \lmtwocstfractwo \tau \Bigl( \sqrt{\overline{E}_{k + 1}} + \sqrt{\overline{E}_{k - 1}} \Bigr) \overline{E}_k + \tfrac{\lmtwocstmaxf}{\sqrt{2}} \tau \Bigl( \sqrt{\overline{E}_{k + 1}} + \sqrt{\overline{E}_k} \Bigr)\,.
			\end{align*}
			Dividing both sides of the preceding equations by $\sqrt{\overline{E}_{k + 1}}$, we arrive at:
			\begin{equation*}
				\sqrt{\overline{E}_{k + 1}} \leq ( 1 + \lmtwocstfracone \tau ) \sqrt{\frac{\overline{E}_k}{\overline{E}_{k + 1}}} \sqrt{\overline{E}_k} + \lmtwocstfractwo \tau \Biggl( 1 + \sqrt{\frac{\overline{E}_{k - 1}}{\overline{E}_{k + 1}}} \Biggr) \overline{E}_k + \tfrac{\lmtwocstmaxf}{\sqrt{2}} \tau \Biggl( 1 + \sqrt{\frac{\overline{E}_k}{\overline{E}_{k + 1}}} \Biggr)\,.
			\end{equation*}
			From this, we deduce that:
			\begin{equation}\label{eq:lemma2_nonlinear_sqrt_E}
				\sqrt{\overline{E}_{k + 1}} \leq ( 1 + \lmtwocstfracone \tau ) \sqrt{\overline{E}_k} + 2 \lmtwocstfractwo \tau \overline{E}_k + \sqrt{2} \lmtwocstmaxf \tau\,.
			\end{equation}
			Let us now consider the case $j = 0$. Then we have $\overline{E}_{k + 1} = \hat{E}_1$, and since $k \geq 1$, it follows that $\overline{E}_k = \hat{E}_1$. Consequently, in this case, inequality \eqref{eq:lemma2_nonlinear_sqrt_E} is trivially fulfilled. \\%
			We now bring \eqref{eq:lemma2_nonlinear_sqrt_E} into the form required by
			\lemref{prop:lemma_RogTsikl2012}. Writing $\mu_k = \sqrt{\overline{E}_k}$, dividing by
			$( 1 + \lmtwocstfracone \tau )^{k+1}$ and setting
			$\overline{\mu}_k = \mu_k / ( 1 + \lmtwocstfracone \tau )^{k}$,
			\eqref{eq:lemma2_nonlinear_sqrt_E} becomes
			\begin{equation*}
				\overline{\mu}_{k + 1} \leq \left( 1 + 2 \lmtwocstfractwo ( 1 + \lmtwocstfracone \tau )^{k - 1} \tau \overline{\mu}_k \right) \overline{\mu}_k + \frac{\sqrt{2} \lmtwocstmaxf}{( 1 + \lmtwocstfracone \tau )^{k + 1}} \tau\,.
			\end{equation*}
			Bounding $\lmtwocstfractwo ( 1 + \lmtwocstfracone \tau )^{k - 1} \leq \lmtwocstnearestexp \coloneqq \lmtwocstfractwo e^{\lmtwocstfracone T}$ and
			$\lmtwocstmaxf ( 1 + \lmtwocstfracone \tau )^{-(k+1)} \leq \lmtwocstmaxf$, and setting
			$\overline{\tau} = 2 \lmtwocstnearestexp \tau$ and $\lmtwocstnearestfinineq = \lmtwocstmaxf / ( \sqrt{2} \lmtwocstnearestexp )$, this reads
			$\overline{\mu}_{k + 1} \leq ( 1 + \overline{\tau} \overline{\mu}_k ) \overline{\mu}_k + \lmtwocstnearestfinineq \overline{\tau}$, so that \lemref{prop:lemma_RogTsikl2012} gives
			\begin{equation}\label{eq:lemma2_nonlinear_unif_est_lamb}
				\overline{\mu}_k \leq \frac{\overline{\mu}}{1 - \overline{\mu} ( 1 + \lmtwocstnearestfinineq ) \overline{t}_k} \quad \text{for} \quad k = 1,2,\ldots,m\,,
			\end{equation}
			where 
			\begin{equation*}
				\overline{\mu} = \max \{ 1, \overline{\mu}_1 \} = \max \Biggl\{ 1, \frac{\sqrt{\hat{E}_1}}{1 + \lmtwocstfracone \tau} \Biggr\} \quad \text{and} \quad \overline{t}_k = k \overline{\tau} = 2 \lmtwocstnearestexp t_k < \frac{1}{\overline{\mu} ( 1 + \lmtwocstnearestfinineq )}\,.
			\end{equation*}
			With the foregoing notation, we have
			\begin{equation}\label{eq:lemma2_nonlinear_est_lam_sqer_E}
				\overline{\mu}_k = \frac{\mu_k}{( 1 + \lmtwocstfracone \tau )^k} \geq \frac{\sqrt{\overline{E}_k}}{e^{\lmtwocstfracone t_k}}\,.
			\end{equation}
			By combining inequality \eqref{eq:lemma2_nonlinear_unif_est_lamb} with \eqref{eq:lemma2_nonlinear_est_lam_sqer_E}, we obtain the following estimate:
			\begin{equation}\label{eq:lemma2_nonlinear_max_ineq_bef_fin}
				\max_{1 \leq i \leq k} \hat{E}_i \leq \frac{\overline{\mu}^2}{\left[ 1 - \sqrt{2} \left( \lmtwocstmaxf + \sqrt{2} \lmtwocstnearestexp \right) \overline{\mu} t_k \right]^2} e^{2 \lmtwocstfracone t_k} \quad \text{for} \quad k = 1,2,\ldots,m\,.
			\end{equation}
			The range $m$ in \eqref{eq:lemma2_nonlinear_max_ineq_bef_fin} is governed by the coefficient of $t_k$ in the denominator, which admits an explicit bound in terms of the data: $\overline{\mu} = \max \{ 1,\overline{\mu}_1 \} \leq \overline{M}$, since $\overline{\mu}_1$ is governed by $\hat{E}_1$, which contains $\hat{V}_1 = \tfrac12 \| \L u_1 \|^2$ with $u_1 = u_0 + \tau v_0 + \tfrac{\tau^2}{2} w_0$ and $w_0 = f_0 - q_0 \L u_0$, so that $\overline{M}$ depends only on $\| \L u_0 \|$, $\| \L^2 u_0 \|$, $\| \L v_0 \|$ and $\| \L f_0 \|$, the quantities supplied by~\eqref{eq:lemma2_data_regularity}. Hence
			\begin{equation}\label{eq:lemma2_nonlinear_max_ineq_fin}
				\max_{1 \leq i \leq k} \hat{E}_i \leq \tfrac{\overline{M}^2}{\left( 1 - \tilde{M} \overline{T} \right)^2} e^{2 \lmtwocstfracone t_k} \quad \text{for} \quad k=1,2,\ldots,\Bigl[ \frac{\overline{T}}{\tau} \Bigr]\,,
			\end{equation}
			with $\tilde{M} = \sqrt{2} \left( \lmtwocstmaxf + \sqrt{2} \lmtwocstnearestexp \right) \overline{M}$ and $\overline{T} = \min \{ T, \varrho / \tilde{M} \}$, where $0 < \varrho < 1$. The minimum keeps $[ \overline{T} / \tau ]$ inside the grid and only shrinks $\overline{T}$, preserving $\tilde{M} \overline{T} < 1$. %
			Thus, \eqref{eq:lemma2_nonlinear_max_ineq_fin} ensures the local uniform boundedness of the sequences of functions $\L^{1/2} \left( \hat{\Delta} u_{k - 1} / \tau \right)$ and $\L u_k$ on the time interval~$\left[ 0,\overline{T} \right]$.
		\end{proof}
		
		\subsection{Convergence Analysis}
		\noindent Writing problem \eqref{eq:main_kirchhoff}--\eqref{eq:initial_conds} with the operator $\L$ at the discrete time levels $t = t_k$, $k = 1,2,\ldots,n - 1$, and rearranging it into the form \eqref{eq:nonlinear_scheme_extended} yields the time-discrete formulation 
		\begin{multline}\label{eq:converg_exact_repres}
			\frac{\hat{\Delta}^2 u \left( t_{k - 1} \right )}{\tau^2} + \tfrac{1}{2}\left[ \alpha_k + \beta_k \left( V \left( t_{k + 1} \right) + V \left( t_{k - 1} \right) \right) \right] \left( \L u \left( t_{k + 1} \right ) + \L u \left( t_{k - 1} \right ) \right) \\
			= f_k + R_{1,k} ( \tau ) + R_{2,k} ( \tau ) + R_{3,k} ( \tau )\,,
		\end{multline}
		where the remainder terms are given by
		\begin{align*}
			R_{1,k} ( \tau ) &= \frac{\hat{\Delta}^2 u \left( t_{k - 1} \right )}{\tau^2} - \frac{\partial^2 u( t_k )}{\partial t^2}\,, \\
			R_{2,k} ( \tau ) &= \left( \tfrac{\alpha_k}{2} + \beta_k V ( t_k ) \right) \hat{\Delta}^2 \left( \L u \left( t_{k - 1} \right ) \right)\,, \\
			R_{3,k} ( \tau ) &= \tfrac{\beta_k}{2} \left( \L u \left( t_{k + 1} \right ) + \L u \left( t_{k - 1} \right ) \right) \hat{\Delta}^2 \left( V \left( t_{k - 1} \right) \right)\,.
		\end{align*}
		Here $V ( t ) = \tfrac12 \| \L^{1/2} u ( t ) \|^2$, as in \eqref{eq:diff_energy}. %
		Subtracting equation \eqref{eq:converg_exact_repres} from equation \eqref{eq:nonlinear_scheme_extended} and, after a straightforward transformation, for the error of the approximate solution $z_k = u ( t_k ) - u_k$, we derive the following equation:
		\begin{equation}\label{eq:converg_error_equation}
			\frac{\hat{\Delta}^2 z_{k - 1}}{\tau^2} + \tfrac{1}{2} \left[ \alpha_k + \beta_k \left( V_{k + 1} + V_{k - 1} \right) \right] \left( \L z_{k + 1} + \L z_{k - 1} \right) = \tilde{g}_k + R_k ( \tau )\,.
		\end{equation}
		Here,
		\begin{align*}
			\tilde{g}_k &= \tfrac{\beta_k}{2} \left[ \left( V_{k + 1} - V \left( t_{k + 1} \right) \right) + \left( V_{k - 1} - V \left( t_{k - 1} \right) \right) \right] \left( \L u \left( t_{k + 1} \right ) + \L u \left( t_{k - 1} \right ) \right)\,, \\
			R_k ( \tau ) &= R_{1,k} ( \tau ) + R_{2,k} ( \tau ) + R_{3,k} ( \tau )\,.
		\end{align*}
		
        The theorem below requires regularity of the solution \emph{beyond} what is needed for the solution to solve the PDE. It states exactly how much such regularity is assumed, not derived. Two additional degrees of smoothness, together with the starting value \eqref{eq:starting_vec1}, give the second order in time, which is the accuracy of the scheme \eqref{eq:nonlinear_semidiscrete} and cannot be exceeded. Throughout, subscripts on $u$ other than the time index denote partial derivatives with respect to the corresponding variables. 
		
		\begin{theorem}\label{prop:conver_thm_nonlinear}
			Assume that the initial–boundary value problem \eqref{eq:main_kirchhoff}--\eqref{eq:initial_conds} is well posed, and let $\left\{ u_k \right\}$ be a solution sequence of the scheme \eqref{eq:nonlinear_semidiscrete} on the interval considered. Suppose moreover that the exact solution $u$ satisfies the following regularity assumptions.
			\begin{enumerate}[label=(\roman*)]
				\item\label{itm_theorem_i} Assume $u_0 \in D ( \L_0 )$, $v_0, w_0 \in C^1 ( \overline{\domain} )$, that $u_x$ and $u_y$ have continuous second temporal derivatives, and that for every $t \in [ 0,T ]$ the source satisfies $f ( t ) \in C^1 ( \overline{\domain} )$ with $f ( t ) |_{\partial \domain} = 0$, its spatial derivatives being bounded uniformly in $t$. Beyond this, impose the operator-domain regularity~\eqref{eq:lemma2_data_regularity} of \lemref{prop:lemma2_nonlinear}, which is what the \apriori{} bounds used below require and which is \emph{not} implied by the conditions just listed; it also places $u_0, v_0, w_0$ and $f( t )$ in the domains of the fractional powers of $\L$, so that the starting value~\eqref{eq:starting_vec1} satisfies the homogeneous Dirichlet condition $u_1 |_{\partial \domain} = 0$ that the scheme imposes --- a compatibility requirement which is not automatic.
				\item\label{itm_theorem_ii} Let the solution $u ( t )$ of the initial–boundary value problem \eqref{eq:main_kirchhoff}--\eqref{eq:initial_conds} be three times continuously differentiable with respect to the temporal variable $t$. Furthermore, the third-order temporal derivative $u_{ttt} ( t )$ is Lipschitz continuous in $t$ as a map into $L^2 ( \domain )$.
				\item\label{itm_theorem_iii} Assume that the second-order spatial derivatives $u_{xx} ( t )$ and $u_{yy} ( t )$ are continuously differentiable with respect to the temporal variable $t$. Moreover, the mixed partial derivatives $u_{xxt} ( t )$ and $u_{yyt} ( t )$ are Lipschitz continuous in $t$, again as maps into $L^2 ( \domain )$. Naming the spatial norm matters: the remainder bounds \eqref{eq:converg_error_remainders} are uniform in $t$ only if the Lipschitz constants are, and an unqualified statement leaves open in which norm that uniformity is meant.
			\end{enumerate}
			There exists $\overline{T} \in \left( 0,T \right]$ such that the error of the approximate solution, defined by $z_k = u ( t_k ) - u_k$, satisfies the following estimates:
			\begin{equation*}
				\max_{1 \leq k \leq m}\left\| \nabla z_k \right\| \leq \cststatementone \tau^2 \quad \text{and} \quad \max_{1 \leq k \leq m} \bigl\| \frac{\hat{\Delta} z_{k - 1}}{\tau} \bigr\| \leq \cststatementtwo \tau^2\,,
			\end{equation*}
			in which $m = \left[ \overline{T} / \tau \right]$, and $\hat{\Delta} z_{k - 1} = z_k - z_{k - 1}$.
		\end{theorem}

        \begin{remark}
            Under the regularity assumptions \ref{itm_theorem_ii} and \ref{itm_theorem_iii} on the solution $u ( t )$ to problem \eqref{eq:main_kirchhoff}--\eqref{eq:initial_conds} imposed in \thmref{prop:conver_thm_nonlinear}, the following estimates for the remainder terms hold (see, \eg, \cite{RogTsikl2014,RogVashZAMM2024}):
		\begin{equation}\label{eq:converg_error_remainders}
			\left\| R_{i,k} ( \tau ) \right\| = \bigO ( \tau^2 )\,,\quad i = 1,2,3\,.
		\end{equation}
		The norm here is the $L^2 ( \domain )$-norm used throughout, which is the only one the energy argument requires: \eqref{eq:converg_error_remainders} enters exactly once, in \eqref{eq:conver_thm_nonlinear_inner_prod_rem}, and there only through the inner product of $R_k ( \tau )$ with a difference of iterates. The cited works state the corresponding bounds in $L^\infty ( \domain )$; since $\domain$ is bounded, those imply the form used here, and stating the weaker form is what allows the Lipschitz hypotheses of \ref{itm_theorem_ii} and \ref{itm_theorem_iii} to be imposed in $L^2 ( \domain )$ rather than uniformly in space.
        \end{remark}

        \begin{proof}[Proof of \Cref{prop:conver_thm_nonlinear}]
            For each $k = 1,2,\ldots,n - 1$, taking the inner product of both sides of equation \eqref{eq:converg_error_equation} with $z_{k + 1} - z_{k - 1} = \hat{\Delta} z_k + \hat{\Delta} z_{k - 1}$ yields
            \begin{multline*}
                \tilde{T}_{k + 1} + \tfrac{1}{2} \left[ \alpha_k + \beta_k \left( V_{k + 1} + V_{k - 1} \right) \right] \tilde{V}_{k + 1} \\
                = \tilde{T}_k + \tfrac{1}{2} \left[ \alpha_k + \beta_k \left( V_{k + 1} + V_{k - 1} \right) \right] \tilde{V}_{k - 1} + \tfrac{1}{2} ( \tilde{g}_k,\hat{\Delta} z_k + \hat{\Delta} z_{k - 1} ) + \tfrac{1}{2} ( R_k ( \tau ),\hat{\Delta} z_k + \hat{\Delta} z_{k - 1} )\,,
            \end{multline*}
            where
            \begin{equation*}
                \tilde{T}_k = \tfrac{1}{2} \Bigl\| \frac{\hat{\Delta} z_{k - 1}}{\tau} \Bigr\|^2\,, \quad \tilde{V}_k = \tfrac{1}{2} \| \L^{1/2} z_k \|^2\,, \quad \text{and} \quad V_k = \tfrac{1}{2} \| \L^{1/2} u_k \|^2\,.
            \end{equation*}
            By adding the same term $\left[ \alpha_k + \beta_k \left( V_{k + 1} + V_{k - 1} \right) \right] \tilde{V}_k / 2$ to both sides of the preceding equation, and defining additional notation $W_k = V_{k + 1} + V_{k - 1} \text{and} \quad \tilde{W}_k = \frac12({\tilde{V}_k + \tilde{V}_{k - 1}})$, we find that
            \begin{multline*}
                \tilde{T}_{k + 1} + ( \alpha_k + \beta_k W_k ) \tilde{W}_{k + 1} \\
                = \tilde{T}_k + ( \alpha_k + \beta_k W_k ) \tilde{W}_k + \tfrac{1}{2} ( \tilde{g}_k,\hat{\Delta} z_k + \hat{\Delta} z_{k - 1} ) + \tfrac{1}{2} ( R_k ( \tau ),\hat{\Delta} z_k + \hat{\Delta} z_{k - 1} )\,.
            \end{multline*}
            Hence, by standard algebraic manipulations, one obtains the following expression:
            \begin{align}\label{eq:conver_thm_nonlinear_tilde_e_expr}
                &\tilde{E}_{k + 1} = \tilde{E}_k + \left( \alpha_k - \alpha_{k - 1} \right) \tilde{W}_k + \left( \beta_k W_k - \beta_{k - 1} W_{k - 1} \right) \tilde{W}_k\nonumber \\
                &\qquad + \tfrac{1}{2} ( \tilde{g}_k,\hat{\Delta} z_k + \hat{\Delta} z_{k - 1} ) + \tfrac{1}{2} ( R_k ( \tau ),\hat{\Delta} z_k + \hat{\Delta} z_{k - 1} )\,,
            \end{align}
            where $\tilde{E}_k = \tilde{T}_k + \left( \alpha_{k - 1} + \beta_{k - 1} W_{k - 1} \right) \tilde{W}_k$.
            Assume, as in the proof of \lemref{prop:lemma2_nonlinear}, that $u_{-1} = u_0$. The first estimate of \lemref{prop:lemma2_nonlinear} together with bound \eqref{eq:lemma2_nonlinear_diff_W} implies
            \begin{equation}\label{eq:conver_thm_nonlinear_diff_w}
                \left| W_k - W_{k - 1} \right| \leq \sqrt{2} \cstmtwo \tau \Bigl( \sqrt{\hat{\mathcal{K}}_{k + 1}} + \sqrt{\hat{\mathcal{K}}_{k - 1}} \Bigr) \leq 2 \cstmtwo \cstmthree \tau\,.
            \end{equation}
            From inequalities \eqref{eq:lemma1_nonlinear_beta_est} and \eqref{eq:conver_thm_nonlinear_diff_w}, together with the fact that $W_k = V_{k + 1} + V_{k - 1} \leq \cstmtwo^2$ provided by \lemref{prop:lemma1_nonlinear}, it follows that
            \begin{align}\label{eq:conver_thm_nonlinear_beta_w_diff}
                \left| \beta_k W_k - \beta_{k - 1} W_{k - 1} \right| &\leq \left| \beta_k - \beta_{k - 1} \right| W_k + \beta_{k - 1} \left| W_k - W_{k - 1} \right|\nonumber \\
                & \leq \thmcstdiffbetaw \tau\,, \quad \text{where} \quad \thmcstdiffbetaw = \cstmtwo \left( \lmonecstbeta \cstmtwo + 2 \cstmthree \beta_{\max} \right)\,.
            \end{align}
            Condition \ref{itm_theorem_iii} of \thmref{prop:conver_thm_nonlinear} implies the bounds
            \begin{equation}\label{eq:conver_thm_nonlinear_norm_ineq_half_l}
                \| \L^{1/2} u ( t ) \| \leq \thmcsthalflnorm \quad \text{and} \quad \left\| \L u ( t ) \right\| \leq \thmcstlnorm\,.
            \end{equation}
            Combining the first inequality in \eqref{eq:conver_thm_nonlinear_norm_ineq_half_l} with \lemref{prop:lemma1_nonlinear} yields the following estimate for the difference:
            \begin{equation}\label{eq:conver_thm_nonlinear_diff_v}
                \left| V_k - V ( t_k ) \right| \leq \Bigl( \sqrt{V_k} + \sqrt{V ( t_k )} \Bigr) \sqrt{\tilde{V}_k} \leq \tfrac{1}{\sqrt{2}} ( \cstmtwo + \thmcsthalflnorm ) \sqrt{\tilde{V}_k}\,.
            \end{equation}
            Applying first the Cauchy–Schwarz inequality and subsequently the triangle inequality, together with bounds \eqref{eq:conver_thm_nonlinear_norm_ineq_half_l} and \eqref{eq:conver_thm_nonlinear_diff_v}, one obtains the following estimate for the inner product appearing in \eqref{eq:conver_thm_nonlinear_tilde_e_expr}, which is associated with $\tilde{g}_k$. Thus, we find that
            \begin{align*}
                \tfrac{1}{2} | ( \tilde{g}_k,\hat{\Delta} z_k + \hat{\Delta} z_{k - 1} ) | &\leq \tfrac{1}{2} \| \tilde{g}_k \| \| \hat{\Delta} z_k + \hat{\Delta} z_{k - 1} \| \\
                &\leq \tfrac{\beta_k}{2} ( \cstmtwo + \thmcsthalflnorm ) \thmcstlnorm \tau \Bigl( \sqrt{\tilde{V}_{k + 1}} + \sqrt{\tilde{V}_{k - 1}} \Bigr) \Bigl( \sqrt{\tilde{T}_{k + 1}} + \sqrt{\tilde{T}_k} \Bigr) \\
                &\leq \tfrac{\beta_k}{2} ( \cstmtwo + \thmcsthalflnorm ) \thmcstlnorm \tau \Bigl( \sqrt{\tilde{V}_{k + 1}} + \sqrt{\tilde{V}_{k - 1}} \Bigr) \Bigl( \sqrt{\tilde{E}_{k + 1}} + \sqrt{\tilde{E}_k} \Bigr)\,.
            \end{align*}
            By the definition of $\tilde{E}_k$ we have $\tilde{E}_k \geq \alpha_{k - 1} \tilde{W}_k \geq \tfrac{\alpha_{\min}}{2} \tilde{V}_k$, hence $\tilde{V}_k \leq 2 \tilde{E}_k / \alpha_{\min}$ and $\tilde{W}_k \leq \tilde{E}_k / \alpha_{\min}$, whence
            \begin{equation}\label{eq:conver_thm_nonlinear_inner_prod_g}
                \tfrac{1}{2} | ( \tilde{g}_k,\hat{\Delta} z_k + \hat{\Delta} z_{k - 1} ) | \leq \thmcstinnerprodgk \tau \Bigl( \sqrt{\tilde{E}_{k + 1}} + \sqrt{\tilde{E}_{k - 1}} \Bigr) \Bigl( \sqrt{\tilde{E}_{k + 1}} + \sqrt{\tilde{E}_k} \Bigr)\,,
            \end{equation}
            where $\thmcstinnerprodgk = \frac{( \cstmtwo + \thmcsthalflnorm ) \thmcstlnorm}{\sqrt{2 \alpha_{\min}}} \beta_{\max}. $
            It remains to estimate the inner product in \eqref{eq:conver_thm_nonlinear_tilde_e_expr} corresponding to the remainder term. To this end, we argue as in the derivation of \eqref{eq:conver_thm_nonlinear_inner_prod_g} and invoke inequality \eqref{eq:converg_error_remainders}. Consequently, we have
            \begin{align}\label{eq:conver_thm_nonlinear_inner_prod_rem}
                \tfrac{1}{2} | ( R_k ( \tau ),\hat{\Delta} z_k + \hat{\Delta} z_{k - 1} ) | &\leq \tfrac{1}{2} \left\| R_k ( \tau ) \right\| ( \| \hat{\Delta} z_k \| + \| \hat{\Delta} z_{k - 1} \| )\nonumber \\
                &\leq \thmcstinnerprodrem \tau^3 \bigl( \sqrt{\tilde{E}_{k + 1}} + \sqrt{\tilde{E}_k} \bigr)\,.
            \end{align}
            Taking into account inequalities \eqref{eq:lemma1_nonlinear_alpha_est}, \eqref{eq:conver_thm_nonlinear_beta_w_diff}, \eqref{eq:conver_thm_nonlinear_inner_prod_g}, and \eqref{eq:conver_thm_nonlinear_inner_prod_rem} in relation \eqref{eq:conver_thm_nonlinear_tilde_e_expr}, we obtain
            \begin{align*}
                &\tilde{E}_{k + 1} \leq \tilde{E}_k + ( \lmonecstalph + \thmcstdiffbetaw ) \tau \tilde{W}_k\nonumber \\
                &\qquad + \thmcstinnerprodgk \tau \Bigl( \sqrt{\tilde{E}_{k + 1}} + \sqrt{\tilde{E}_{k - 1}} \Bigr) \Bigl( \sqrt{\tilde{E}_{k + 1}} + \sqrt{\tilde{E}_k} \Bigr) + \thmcstinnerprodrem \tau^3 \Bigl( \sqrt{\tilde{E}_{k + 1}} + \sqrt{\tilde{E}_k} \Bigr)\,.
            \end{align*}
            Inserting $\tilde{W}_k \leq \tilde{E}_k / \alpha_{\min}$ and setting
            \begin{equation*}
                \thmcstsumoverczero = \tfrac{\lmonecstalph + \thmcstdiffbetaw}{\alpha_{\min}}\,,\qquad \thmcstmaxtwoconst = \max \{ \thmcstinnerprodgk,\thmcstinnerprodrem \}\,,
            \end{equation*}
            the recurrence becomes $\tilde{E}_{k + 1} \leq \tilde{E}_k + \delta_k$ with
            $\delta_k = \thmcstsumoverczero \tau \tilde{E}_k + \thmcstmaxtwoconst \tau ( \sqrt{\tilde{E}_{k + 1}} + \sqrt{\tilde{E}_k} ) ( \tau^2 + \sqrt{\tilde{E}_{k + 1}} + \sqrt{\tilde{E}_{k - 1}} )$, and unrolling it from $1$ to $k$ gives
            \begin{equation*}
                \tilde{E}_{k + 1} \leq \tilde{E}_1 + \thmcstsumoverczero \tau \sum_{i = 1}^{k} \tilde{E}_i + \thmcstmaxtwoconst \tau \sum_{i = 1}^{k} \Bigl( \sqrt{\tilde{E}_{i + 1}} + \sqrt{\tilde{E}_i} \Bigr) \Bigl( \tau^2 + \sqrt{\tilde{E}_{i + 1}} + \sqrt{\tilde{E}_{i - 1}} \Bigr)\,.
            \end{equation*}
            Assume that $\tilde{E}_i$, for $i = 1,2,\ldots,k + 1$, attains its maximum at the index $j$. That is,
            \begin{equation*}
                \tilde{E}_j = \max_{1 \leq i \leq k + 1} \tilde{E}_i\,.
            \end{equation*}
            Under this assumption, if $j > 1$, then the following holds:
            \begin{equation*}
                \tilde{E}_j \leq \tilde{E}_1 + \thmcstsumoverczero \tau \sum_{i = 1}^{j - 1} \tilde{E}_i + \thmcstmaxtwoconst \tau \sum_{i = 1}^{j - 1} \Bigl( \sqrt{\tilde{E}_{i + 1}} + \sqrt{\tilde{E}_i} \Bigr) \Bigl( \tau^2 + \sqrt{\tilde{E}_{i + 1}} + \sqrt{\tilde{E}_{i - 1}} \Bigr)\,.
            \end{equation*}
            Dividing both sides of the preceding inequality by $\sqrt{\tilde{E}_j}$, we obtain ({\cf} \cite{RogTsikl2014,RogVashZAMM2024}):
            \begin{align*}
                \sqrt{\tilde{E}_j} &\leq \sqrt{\tilde{E}_1} + \thmcstsumoverczero \tau \sum_{i = 1}^{j - 1} \sqrt{\tilde{E}_i} + 2 \thmcstmaxtwoconst \tau \sum_{i = 1}^{j - 1} \Bigl( \tau^2 + \sqrt{\tilde{E}_{i + 1}} + \sqrt{\tilde{E}_{i - 1}} \Bigr) \\
                &\leq \sqrt{\tilde{E}_1} + \thmcstsumoverczero \tau \sum_{i = 1}^{k} \sqrt{\tilde{E}_i} + 2 \thmcstmaxtwoconst \tau \sum_{i = 1}^{k} \Bigl( \tau^2 + \sqrt{\tilde{E}_{i + 1}} + \sqrt{\tilde{E}_{i - 1}} \Bigr)\,.
            \end{align*}
            Since $\tilde{E}_{k + 1} \leq \tilde{E}_j$, we infer that
            \begin{equation}\label{eq:conver_thm_nonlinear_sqrt_tilde_E_k_ineq}
                \sqrt{\tilde{E}_{k + 1}} \leq \sqrt{\tilde{E}_1} + \thmcstsumoverczero \tau \sum_{i = 1}^{k} \sqrt{\tilde{E}_i} + 2 \thmcstmaxtwoconst \tau \sum_{i = 1}^{k} \Bigl( \tau^2 + \sqrt{\tilde{E}_{i + 1}} + \sqrt{\tilde{E}_{i - 1}} \Bigr)\,.
            \end{equation}
            For $j = 1$ the inequality is immediate, since then $\tilde{E}_{k+1} \leq \tilde{E}_1$ and the two sums on the right are non-negative.
            Rearranging \eqref{eq:conver_thm_nonlinear_sqrt_tilde_E_k_ineq}, bounding each of the three sums by the full sum $\sum_{i=1}^{k} \sqrt{\tilde{E}_i}$ and using $\tau^2 t_k \leq \overline{T} \tau^2$, one obtains
            \begin{equation*}
                ( 1 - 2 \thmcstmaxtwoconst \tau ) \sqrt{\tilde{E}_{k + 1}} \leq \Bigl( \sqrt{\tilde{E}_1} + 2 \thmcstmaxtwoconst \tau \sqrt{\tilde{E}_0} + 2 \overline{T} \thmcstmaxtwoconst \tau^2 \Bigr) + ( \thmcstsumoverczero + 4 \thmcstmaxtwoconst ) \tau \sum_{i = 1}^{k} \sqrt{\tilde{E}_i}\,.
            \end{equation*}
            Hence, provided that $\thmcstmaxtwoconst \tau < 1 / 2$ and setting the constants
            \begin{equation*}
                \tilde{\varepsilon} = \frac{\sqrt{\tilde{E}_1} + 2 \thmcstmaxtwoconst \tau \sqrt{\tilde{E}_0} + 2 \overline{T} \thmcstmaxtwoconst \tau^2}{1 - 2 \thmcstmaxtwoconst \tau} \quad \text{and} \quad \tilde{c} = \frac{\thmcstsumoverczero + 4 \thmcstmaxtwoconst}{1 - 2 \thmcstmaxtwoconst \tau}\,,
            \end{equation*}
            the following holds:
            \begin{equation}\label{eq:conver_thm_nonlinear_before_unrolling}
                \sqrt{\tilde{E}_{k + 1}} \leq \tilde{\varepsilon} + \tilde{c} \tau \sum_{i = 1}^{k} \sqrt{\tilde{E}_i}\,.
            \end{equation}
            Since $\sqrt{\tilde{E}_1} \leq \tilde{\varepsilon}$, repeated application of \eqref{eq:conver_thm_nonlinear_before_unrolling} gives
            \begin{equation}\label{eq:conver_thm_nonlinear_sqrt_E_k_plus}
                \sqrt{\tilde{E}_{k + 1}} \leq \left( 1 + \tilde{c} \tau \right)^k \tilde{\varepsilon} \leq e^{\tilde{c} t_k} \tilde{\varepsilon} \leq e^{\tilde{c} \overline{T}} \tilde{\varepsilon}\,.
            \end{equation}
            It remains to estimate $\tilde{\varepsilon}$. With $u_{-1} = u_0$ we have $\tilde{E}_0 = 0$, and $\tilde{E}_1 = \tilde{T}_1 + ( \alpha_0 + \beta_0 W_0 ) \tilde{W}_1$ by the definition of $\tilde{E}_k$. Under the regularity assumptions \ref{itm_theorem_i}–\ref{itm_theorem_iii},
            \begin{equation*}
                \tilde{T}_1 = \tfrac{1}{2} \Bigl\| \frac{\hat{\Delta} z_0}{\tau} \Bigr\|^2 = \frac{1}{2 \tau^2} \left\| u ( t_1 ) - u_1 \right\|^2 = \frac{1}{2 \tau^2} \Bigl\| \int\limits_{0}^{\tau} \!\! \int\limits_{0}^{t} \!\! \int\limits_{0}^{s} u_{ttt} ( \xi ) \mathrm{d}\xi \mathrm{d}s \mathrm{d}t \Bigr\|^2 \leq \thmcstesttone \tau^4\,,
            \end{equation*}
            and
            \begin{align*}
                \tilde{W}_1 = \tfrac{1}{2} \tilde{V}_1 &= \frac{1}{4} \| \L^{1/2} z_1 \|^2 = \tfrac{1}{4} \Bigl\| \int\limits_{0}^{\tau} \L^{1/2} \left( u_t ( t ) - u_t ( 0 ) \right) \mathrm{d}t - \tau^2 \tfrac12 \L^{1/2} w_0 \Bigr\|^2 \leq \thmcstestwone \tau^4\,.
            \end{align*}
            With \lemref{prop:lemma1_nonlinear} these give $\tilde{E}_1 \leq ( \thmcstesttone + ( \alpha ( 0 ) + \beta ( 0 ) \cstmtwo^2 ) \thmcstestwone ) \tau^4$, hence $\tilde{\varepsilon} \leq \thmtildevareps \tau^2$ and, by \eqref{eq:conver_thm_nonlinear_sqrt_E_k_plus}, $\sqrt{\tilde{E}_{k + 1}} \leq \thmtildeEkplusone \tau^2$. The estimates stated in \thmref{prop:conver_thm_nonlinear} follow.
        \end{proof}

        \begin{corollary}\label{prop:corollary_one}
            The error $z_k = u ( t_k ) - u_k$ associated with the approximate solutions of problem \eqref{eq:main_kirchhoff}--\eqref{eq:initial_conds} is second order with respect to the temporal grid spacing $\tau$ on a local interval $\left[ 0,\overline{T} \right]$, where $\overline{T} \leq T$. That is,
            \begin{equation*}
                \max_{1 \leq k \leq m} \left\| z_k \right\| = \bigO ( \tau^2 )\,, \quad \text{with} \quad m = \left[ \frac{\overline{T}}{\tau} \right]\,.
            \end{equation*}
        \end{corollary}

        \begin{proof}
        Immediate from $z_k = z_0 + \tau \sum_{i = 1}^{k} \hat{\Delta} z_{i - 1} / \tau$ and the second bound of \thmref{prop:conver_thm_nonlinear}.
        \end{proof}

        \begin{corollary}\label{prop:corollary_two}
            The error associated with the first-order time derivative of the solution $u ( t )$ to problem \eqref{eq:main_kirchhoff}--\eqref{eq:initial_conds} at the discrete time level $t_k = k \tau$, obtained by applying the central finite difference formula to its approximation $u_k$, is second-order accurate in time. More precisely, on a local time interval,
            \begin{equation*}
                \max_{1 \leq k \leq m - 1} \bigl\| u_t ( t_k ) - \frac{u_{k + 1} - u_{k - 1}}{2 \tau} \bigr\| = \bigO ( \tau^2 )\,.
            \end{equation*}
        \end{corollary}

        \begin{proof}
        We observe that the error can be represented as the sum of the following two terms:
        \begin{equation}\label{eq:corol_deriv_approx_repres}
            \Bigl( u_t ( t_k ) - \frac{u \left( t_{k + 1} \right ) - u \left( t_{k - 1} \right )}{2 \tau} \Bigr) + \tfrac{1}{2} \Bigl( \frac{\hat{\Delta} z_k}{\tau} + \frac{\hat{\Delta} z_{k - 1}}{\tau} \Bigr)\,.
        \end{equation}
        The first term in expression \eqref{eq:corol_deriv_approx_repres} corresponds to the remainder term of the standard second-order central difference approximation of the first derivative and, under assumption \ref{itm_theorem_ii} imposed in \thmref{prop:conver_thm_nonlinear}, is indeed $\bigO ( \tau^2 )$. Furthermore, for the norm of the second term in \eqref{eq:corol_deriv_approx_repres}, it follows from \thmref{prop:conver_thm_nonlinear} that it is also $\bigO ( \tau^2 )$.
        \end{proof}

        Evaluating the total energy \eqref{eq:total_energy} of the homogeneous problem with constant coefficients at $t = t_k$ and replacing the kinetic part by the central difference computed from the nonlinear scheme \eqref{eq:nonlinear_semidiscrete_homog}, we obtain, for $k = 1,2,\ldots,n-1$, the discrete approximation
        \begin{equation*}
            E_k^{\ast} = \mathcal{K}_k^{\ast} + \alpha V_k + \beta V_k^2\,,\qquad \mathcal{K}_k^{\ast} = \tfrac{1}{2} \Bigl\| \frac{u_{k + 1} - u_{k - 1}}{2 \tau} \Bigr\|^2\,.
        \end{equation*}
        
        \begin{corollary}\label{cor:energy_approximation}
            For the homogeneous problem with constant coefficients $\alpha,\beta\in\mathbb{R}_+$, the difference between the total energy $E$ of \eqref{eq:total_energy} and its discrete approximation $E_k^\ast$ is of order $\bigO ( \tau^2 )$ with respect to the time-step size. Consequently,
            \begin{equation}\label{eq:rem_max_diff_tot_energy}
                \max_{1 \leq k \leq m - 1} \left| E ( t_k ) - E_k^{\ast} \right| = \bigO ( \tau^2 )\,.
            \end{equation}
        \end{corollary}
        \begin{proof}
            It is evident that
            \begin{equation}\label{eq:rem_tot_ener_diff_ineq}
                \left| E ( t_k ) - E_k^{\ast} \right| \leq \left| \mathcal{K} ( t_k ) - \mathcal{K}_k^{\ast} \right| + \left( \alpha + \beta \left(  V ( t_k ) + V_k \right) \right) \left| V ( t_k ) - V_k \right|\,.
            \end{equation}
            By \lemref{prop:lemma1_nonlinear} and \corref{prop:corollary_two},
            \begin{align}\label{eq:rem_tot_diff_kinet}
                \left| \mathcal{K} ( t_k ) - \mathcal{K}_k^{\ast} \right| &= \tfrac{1}{2} \Bigl| \left\| u_t ( t_k ) \right\|^2 - \Bigl\| \frac{u_{k + 1} - u_{k - 1}}{2 \tau} \Bigr\|^2 \Bigr| \nonumber \\
                &\leq \tfrac{1}{2} \bigl( \left\| u_t ( t_k ) \right\| + \tfrac{1}{\sqrt{2}} \bigl( \sqrt{\mathcal{K}_{k + 1}} + \sqrt{\mathcal{K}_k} \bigr) \bigr) \Bigl\| u_t ( t_k ) - \frac{u_{k + 1} - u_{k - 1}}{2 \tau} \Bigr\| \leq \rmkthmcstkinetic \tau^2\,.
            \end{align}
            By \thmref{prop:conver_thm_nonlinear} and inequality \eqref{eq:conver_thm_nonlinear_diff_v}, it follows that
            \begin{equation}\label{eq:rem_tot_diff_potent}
                \left| V ( t_k ) - V_k \right| \leq \tfrac{1}{2} ( \cstmtwo + \thmcsthalflnorm ) \cststatementone \tau^2\,.
            \end{equation}
            Inserting \eqref{eq:rem_tot_diff_kinet} and \eqref{eq:rem_tot_diff_potent} into \eqref{eq:rem_tot_ener_diff_ineq} gives \eqref{eq:rem_max_diff_tot_energy}.
        \end{proof}
		
		\section{Solvability of a Nonlinear Equation Arising from Temporal Discretization}\label{sec:solvability}

        \subsection{Existence and Uniqueness of a Solution}\label{subsec:iteration_existence}

		To construct an iterative procedure for the Crank–Nicolson-type nonlinear scheme \eqref{eq:nonlinear_scheme_extended}, we recast it in the equivalent form 
		\begin{equation}\label{eq:nonlinear_equation_v}
			v_{k + 1} + \frac{\tau^2}{2}\bigl( \alpha_k + \tfrac{\beta_k}{2} \bigl( \| \L^{1/2} u_{k + 1} \|^2 + \| \L^{1/2} u_{k - 1} \|^2 \bigr) \bigr) \L v_{k + 1} = g_k\,,
		\end{equation}
		with
		\begin{equation*}
			v_{k + 1} = \tfrac12({u_{k + 1} + u_{k - 1}})\,,\quad g_k = \frac{\tau^2}{2} f_k + u_k\,.
		\end{equation*}
		\noindent The auxiliary variable $v_{k+1}$ introduced here is local to this section and
		is unrelated to the initial velocity $v_0$ of~\eqref{eq:initial_conds}; it is the
		semi-discrete counterpart of the quantity denoted $v_h^{n+1}$ in the fully discrete
		scheme of~\cite{KirchhoffFEM}, where the subscript $h$ keeps the two apart.
		Let the operator be defined by
		\begin{equation*}
			\T_k = \I + \frac{\tau^2}{2}\bigl( \alpha_k + \tfrac{\beta_k}{2} \bigl( \| \L^{1/2} u_{k + 1} \|^2 + \| \L^{1/2} u_{k - 1} \|^2 \bigr) \bigr) \L\,.
		\end{equation*}
		As a consequence, equation \eqref{eq:nonlinear_equation_v} can be written in the form:
		\begin{equation}\label{eq:operator_equation_v}
			\T_k v_{k + 1} = g_k\,,\quad g_k \in L^2 ( \domain )\,.
		\end{equation}
		To solve equation \eqref{eq:operator_equation_v}, the following iteration is used
		\begin{equation}\label{eq:oper_iter}
			\T_{k,\ell} v_{k + 1}^{( \ell + 1 )} = g_k\,,\quad k = 1,2,\ldots,n - 1\,,\quad \ell = 0,1,\ldots\,,
		\end{equation}
		where
		\begin{align*}
			\T_{k,\ell} &= \I + \frac{\tau^2}{2}\bigl( \alpha_k + \tfrac{\beta_k}{2} \bigl( \| \L^{1/2} u_{k + 1}^{( \ell )} \|^2 + \| \L^{1/2} u_{k - 1} \|^2 \bigr) \bigr) \L\,, \\
			v_{k + 1}^{( 0 )} &= \tfrac12({u_k + u_{k - 1}})\,,\quad u_{k + 1}^{( 0 )} = u_k\,,\quad u_{k + 1}^{( \ell )} = 2 v_{k + 1}^{( \ell )} - u_{k - 1}\,.
		\end{align*}
		\noindent The initial iterate $v_{k+1}^{(0)}$ is fixed here so that the relation $u_{k+1}^{(\ell)} = 2 v_{k+1}^{(\ell)} - u_{k-1}$ holds for every $\ell \geq 0$, including $\ell=0$; this makes the estimates below uniform in the iteration index.
		In the representation of the operator $\T_{k,\ell}$, the presence of the small parameter $\tau$ enables us to demonstrate the convergence of the iterative process \eqref{eq:oper_iter}. As an initial step, we establish the uniform boundedness of the sequence of vectors $w_{k + 1}^{( \ell + 1 )} = \L^{1/2} v_{k + 1}^{( \ell + 1 )}$ obtained from iteration \eqref{eq:oper_iter}. This property is important, as it provides the basis for proving the convergence of the iterative scheme.
		
		Defining $v_{k + 1}^{( \ell + 1 )}$ using equation \eqref{eq:oper_iter} and applying the operator $\L^{1/2}$, we obtain:
		\begin{equation}\label{eq:iter_w}
			w_{k + 1}^{( \ell + 1 )} = \T_{k,\ell}^{-1} \L^{1/2} g_k\,.
		\end{equation}
		Hence, we have:
		\begin{align}\label{eq:vec_w}
			\| w_{k + 1}^{( \ell + 1 )} \| &= \| \T_{k,\ell}^{-1} \L^{1/2} g_k \| = \bigl\| \T_{k,\ell}^{-1} \L^{1/2} \bigl( \frac{\tau^2}{2} f_k + u_k \bigr) \bigr\|\nonumber \\
			&\leq \frac{\tau^2}{2} \| \T_{k,\ell}^{-1} \L^{1/2} \| \| f_k \| + \| \T_{k,\ell}^{-1} \| \| \L^{1/2} u_k \|\,.
		\end{align}
		The following estimates are readily derived:
		\begin{equation}\label{eq:operator_ineq}
			\| \T_{k,\ell}^{-1} \| \leq 1\,,\quad \| \T_{k,\ell}^{-1} \L^{1/2} \| \leq \tau^{-1}\tfrac{1}{\sqrt{2 \alpha_{\min}}}\,,
		\end{equation}
		Here and below, operator expressions such as $\T_{k,\ell}^{-1} \L^{a}$ are understood through
		the spectral calculus of the self-adjoint, positive definite $\L$: since $\T_{k,\ell}$ is a
		positive function of $\L$, each such expression is the bounded multiplier
		$\lambda \mapsto \lambda^{a} ( 1 + \tfrac{\tau^2}{2} q \lambda )^{-1}$, and
		\eqref{eq:operator_ineq} is the supremum of that multiplier over the spectrum.

		Recall that the sequences $\L^{1/2} u_k$ and $\L u_k$ are uniformly bounded ({\cf} \lemref{prop:lemma1_nonlinear} and \lemref{prop:lemma2_nonlinear}). 
		If we substitute the bounds \eqref{eq:operator_ineq} and \eqref{eq:unif_bound_half_lu} into inequality \eqref{eq:vec_w}, and use that the temporal grid size satisfies $\tau = T / n \leq T$, we infer that:
		\begin{equation}\label{eq:unif_iter_vec}
			\| w_{k + 1}^{( \ell + 1 )} \| = \| \L^{1/2} v_{k + 1}^{( \ell + 1 )} \| \leq \cstmfive\,,\quad \cstmfive \coloneqq \tfrac{T}{2 \sqrt{2 \alpha_{\min}}} \max_{1 \leq i \leq n} \left\| f_i \right\| + \cstmtwo\,\quad ( \ell = 0,1,\ldots )\,.
		\end{equation}
		Thus, we show the uniform boundedness of the vectors $w_{k + 1}^{( \ell + 1 )}$ obtained from iteration \eqref{eq:oper_iter}. Since $\cstmfive \geq \cstmtwo$, the bound \eqref{eq:unif_iter_vec} also covers the initial iterate, $\| \L^{1/2} v_{k+1}^{(0)} \| \leq \cstmtwo \leq \cstmfive$, and therefore holds for all $\ell \geq 0$. Note further that the factor $\tau$ produced by \eqref{eq:operator_ineq} has deliberately been replaced by the larger, $\tau$-independent factor $T$: the constant $\cstmfive$, and every constant built from it below, must not depend on the step size, since they enter the step-size restriction of \thmref{thm:iteration_convergence}. As $\tau \leq T$, the bound is not affected.

		We next show that consecutive iterates contract. Note that this step, unlike the
		error estimate of \Cref{subsec:iteration_convergence} below, does not refer to a solution
		of \eqref{eq:operator_equation_v}. Applying $\L^{1/2}$ to \eqref{eq:oper_iter} at two consecutive
		iteration indices and subtracting, we find with \eqref{eq:iter_w} that
		\begin{align}\label{eq:iter_error_init}
			Y_{k + 1}^{( \ell + 1 )} &\coloneqq w_{k + 1}^{( \ell + 1 )} - w_{k + 1}^{( \ell )} = ( \T_{k,\ell}^{-1} - \T_{k,\ell - 1}^{-1} ) \L^{1/2} g_k = \T_{k,\ell}^{-1} ( \T_{k,\ell - 1} - \T_{k,\ell} ) \T_{k,\ell - 1}^{-1} \L^{1/2} g_k\nonumber \\
			&=  \tau^2\tfrac{\beta_k}{4} \T_{k,\ell}^{-1} ( \| \L^{1/2} u_{k + 1}^{( \ell - 1 )} \|^2 - \| \L^{1/2} u_{k + 1}^{( \ell )} \|^2 ) \L \T_{k,\ell - 1}^{-1} \L^{1/2} g_k\,.
		\end{align}
		Taking into account estimates \eqref{eq:unif_bound_half_lu} and \eqref{eq:unif_iter_vec} for the difference of squares appearing in the brackets of equation \eqref{eq:iter_error_init}, and using $u_{k+1}^{(\ell)} = 2 v_{k+1}^{(\ell)} - u_{k-1}$, we obtain the following inequality:
		\begin{align}\label{eq:bracket_term}
			&| \| \L^{1/2} u_{k + 1}^{( \ell - 1 )} \|^2 - \| \L^{1/2} u_{k + 1}^{( \ell )} \|^2 | \leq ( \| \L^{1/2} u_{k + 1}^{( \ell - 1 )} \| + \| \L^{1/2} u_{k + 1}^{( \ell )} \| ) \| \L^{1/2} ( u_{k + 1}^{( \ell - 1 )} - u_{k + 1}^{( \ell )} ) \|\nonumber \\
			&\qquad \leq ( 2\| \L^{1/2} v_{k + 1}^{( \ell - 1 )}  \| + 2\| \L^{1/2} v_{k + 1}^{( \ell )}  \| + 2\| \L^{1/2} u_{k - 1}  \| ) \cdot 2 \| \L^{1/2} ( v_{k + 1}^{( \ell - 1 )} - v_{k + 1}^{( \ell )} ) \|\nonumber \\
			&\qquad \leq 4 ( \cstmtwo + 2 \cstmfive ) \| Y_{k + 1}^{( \ell )} \| = 4 \cstmseight \| Y_{k + 1}^{( \ell )} \|\,,\quad \cstmseight = \cstmtwo + 2 \cstmfive\,.
		\end{align}
		In representation \eqref{eq:iter_error_init}, it remains to estimate the term $\T_{k,\ell}^{-1} \L \T_{k,\ell - 1}^{-1} \L^{1/2} g_k$. By applying estimates \eqref{eq:operator_ineq} and \eqref{eq:unif_bound_lu}, the following inequality holds:
		\begin{align}\label{eq:est_m5}
			\| \T_{k,\ell}^{-1} \L \T_{k,\ell - 1}^{-1} \L^{1/2} g_k \| &= \| ( \T_{k,\ell}^{-1} \L^{1/2} ) ( \T_{k,\ell - 1}^{-1} \L g_k ) \| \leq \| \T_{k,\ell}^{-1} \L^{1/2} \| \| \T_{k,\ell - 1}^{-1} \L g_k \|\nonumber \\
			&\leq \tau^{-1}\tfrac{1}{\sqrt{2 \alpha_{\min}}} ( \tau^2\tfrac12 \| \T_{k,\ell - 1}^{-1} \L \| \| f_k \| + \| \T_{k,\ell - 1}^{-1} \| \| \L u_k \| )\nonumber \\
			&\leq \tau^{-1}\tfrac{1}{\sqrt{2 \alpha_{\min}}} ( \tfrac{1}{\alpha_{\min}} \| f_k \| + \| \L u_k \| )\nonumber \\
			&\leq \tau^{-1} \cstmseven\,,\quad \cstmseven = \tfrac{1}{\sqrt{2 \alpha_{\min}}} ( \tfrac{1}{\alpha_{\min}} \max_{1 \leq i \leq n} \| f_i \| + \cstmfour )\,.
		\end{align}
		From \eqref{eq:iter_error_init}, and taking into account inequalities \eqref{eq:bracket_term} and \eqref{eq:est_m5}, we deduce that:
		\begin{equation}\label{eq:final_iter}
			\| Y_{k + 1}^{( \ell + 1 )} \| \leq \tau \cstmseight \cstmseven \beta_{\max} \| Y_{k + 1}^{( \ell )} \|\,\,.
		\end{equation}
		\noindent The constants $\cstmtwo$, $\cstmfour$, $\cstmfive$, $\cstmseven$ and
		$\cstmseight = \cstmtwo + 2 \cstmfive$ entering \eqref{eq:final_iter} are determined by the
		data of problem \eqref{eq:main_kirchhoff}--\eqref{eq:initial_conds} alone, and in particular
		by none of $\tau$, the time level $k$ and the number $n$ of time steps. Setting
		\begin{equation}\label{eq:tau_zero_iteration}
			\tau_0 \coloneqq ( \cstmseight \cstmseven \beta_{\max} )^{-1}\,,
		\end{equation}
		\noindent \eqref{eq:final_iter} states that consecutive iterates contract with factor 
		$\rho = \tau / \tau_0 < 1$ whenever $0 < \tau < \tau_0$, that is,
        \begin{equation*}
            \| w_{k + 1}^{( \ell + 2 )} - w_{k + 1}^{( \ell + 1 )} \| \leq \rho \| w_{k + 1}^{( \ell + 1 )} - w_{k + 1}^{( \ell )} \| \leq \rho^{\ell+1} \| w_{k + 1}^{( 1 )} - w_{k + 1}^{( 0 )} \|\,, \quad \ell = 0,1,\ldots\,,
        \end{equation*}
        with $w_{k + 1}^{( \ell )} = \L^{1/2} v_{k + 1}^{( \ell )}$. Summing the geometric series shows
        in the standard way that $\{ w_{k+1}^{(\ell)} \}_\ell$ is a Cauchy sequence in $L^2$, hence
        convergent; and since $\L$ is self-adjoint and positive definite, $\{ v_{k+1}^{(\ell)} \}_\ell$
        is Cauchy as well, so that there is $v^{\ast}$ with $v_{k + 1}^{( \ell )} \to v^{\ast}$ and
        $\L^{1/2} v_{k + 1}^{( \ell )} \to \L^{1/2} v^{\ast}$ as $\ell \to \infty$.

        We now show that the vector $v^{\ast}$ is a solution of the nonlinear equation \eqref{eq:nonlinear_equation_v}. To this end, the convergence of the vector sequences $v_{k + 1}^{( \ell )}$ and $\L^{1/2} v_{k + 1}^{( \ell )}$ alone is insufficient; it is additionally necessary to establish the convergence of the sequence $\L v_{k + 1}^{( \ell )}$. From equation \eqref{eq:nonlinear_equation_v}, we have:
        \begin{align*}
            &\| \L v_{k + 1}^{( \ell + p + 1 )} - \L v_{k + 1}^{( \ell + 1 )} \| \leq \tau^{-2}\tfrac{2}{\alpha_{\min} } \| v_{k + 1}^{( \ell + p + 1 )} - v_{k + 1}^{( \ell + 1 )} \| \\
            &\qquad +\tau^{-2}\tfrac{2 \beta_{\max}}{\alpha_{\min}^2 } ( \| \L^{1/2} u_{k + 1}^{( \ell + p )} \| + \| \L^{1/2} u_{k + 1}^{( \ell )} \| ) \| \L^{1/2} v_{k + 1}^{( \ell + p )} - \L^{1/2} v_{k + 1}^{( \ell )} \| ( \| g_k \| + \| v_{k + 1}^{( \ell + 1 )} \| )\,.
        \end{align*}
        It follows from this inequality that $\L v_{k + 1}^{( \ell )}$ forms a Cauchy sequence. From this, since $\L$ is a self-adjoint positive definite operator, it follows that the sequence $\L v_{k + 1}^{( \ell )}$ converges to $\L v^{\ast}$ as $\ell \to \infty$. Passing to the limit as $\ell \to \infty$ in \eqref{eq:oper_iter}, we obtain $\T_k v^{\ast} = g_k$, with $u_{k + 1} = 2 v^{\ast} - u_{k - 1}$. By this, we have shown the solvability of the nonlinear equation \eqref{eq:nonlinear_equation_v}.

        We now establish the uniqueness of the solution to equation \eqref{eq:operator_equation_v}. Assume that equation \eqref{eq:operator_equation_v} admits another solution $\tilde{v}_{k + 1} \neq v_{k + 1}$, satisfying $\widetilde{\T}_k \tilde{v}_{k + 1} = g_k$, where the operator $\widetilde{\T}_k$ is obtained from $\T_k$ by replacing $u_{k + 1}$ with $\tilde{u}_{k + 1} = 2 \tilde{v}_{k + 1} - u_{k - 1}$. Both $v_{k+1}$ and $\tilde{v}_{k+1}$ satisfy an equation of the form \eqref{eq:operator_equation_v}, so the computation \eqref{eq:vec_w}--\eqref{eq:unif_iter_vec} applies to each of them and yields $\| \L^{1/2} v_{k+1} \|, \| \L^{1/2} \tilde{v}_{k+1} \| \leq \cstmfive$, whence $\| \L^{1/2} u_{k+1} \|, \| \L^{1/2} \tilde{u}_{k+1} \| \leq 2 \cstmfive + \cstmtwo = \cstmseight$. Repeating the reasoning used to establish inequality \eqref{eq:final_iter}, we derive:
        \begin{equation}\label{eq:uniq_sol_first}
            \| \L^{1/2} v_{k + 1} - \L^{1/2} \tilde{v}_{k + 1} \| \leq \tilde{\rho} \| \L^{1/2} v_{k + 1} - \L^{1/2} \tilde{v}_{k + 1} \|\,, \quad \tilde{\rho} = \tau \cstmseight \cstmseven \beta_{\max} {}= \frac{\tau}{\tau_0}\,.
        \end{equation}
        If $\tau$ is chosen so that $0 < \tau < \tau_0$, then $\tilde{\rho} < 1$ and relation \eqref{eq:uniq_sol_first} implies that $\L^{1/2} \left( v_{k + 1} - \tilde{v}_{k + 1} \right) = 0$. Hence, since $\L$ is a self-adjoint positive definite operator, it follows that $v_{k + 1} = \tilde{v}_{k + 1}$.

        \subsection{Geometric Convergence of the Iterative Method}\label{subsec:iteration_convergence}

        \noindent Now that the solution $v_{k+1}$ of \eqref{eq:operator_equation_v} is known
        to exist and to be unique, we may measure the iteration error against it. Set
        $Z_{k+1}^{(\ell+1)} = w_{k+1} - w_{k+1}^{(\ell+1)}$ with $w_{k+1} = \L^{1/2} v_{k+1}$. From
        \eqref{eq:operator_equation_v} we get $w_{k + 1} = \T_k^{-1} \L^{1/2} g_k$, and
        subtracting this from \eqref{eq:iter_w} reproduces the computation
        \eqref{eq:iter_error_init}--\eqref{eq:final_iter} with $\T_{k,\ell-1}$ replaced by $\T_k$
        and $u_{k+1}^{(\ell-1)}$ by $u_{k+1}$. Nothing in that computation depends on which member
        of the family the second operator is: \eqref{eq:est_m5} uses only
        \eqref{eq:operator_ineq} and \eqref{eq:unif_bound_lu}, both valid for every operator
        $\T_{k,\cdot}$ and for $\T_k$, and the difference of squares is again bounded by
        $4 \cstmseight \| Z_{k + 1}^{( \ell )} \|$, since
        $\| \L^{1/2} u_{k+1}^{(\ell)} \|, \| \L^{1/2} u_{k+1} \| \leq \cstmseight$ --- for the
        iterate by \eqref{eq:unif_iter_vec} with $u_{k+1}^{(\ell)} = 2 v_{k+1}^{(\ell)} - u_{k-1}$,
        and for the solution by the bound $2 \cstmfive + \cstmtwo = \cstmseight$ derived in the
        uniqueness argument from \eqref{eq:operator_equation_v} itself. The \apriori{} bound
        \eqref{eq:unif_bound_half_lu} must \emph{not} be invoked for $u_{k+1}$ here: it is a
        property of the time-discrete solution sequence, and $u_{k+1}$ is precisely the object
        whose existence is at issue. Hence
        \begin{equation}\label{eq:final_iter_error}
            \| Z_{k + 1}^{( \ell + 1 )} \| \leq \tau \cstmseight \cstmseven \beta_{\max} \| Z_{k + 1}^{( \ell )} \| = \frac{\tau}{\tau_0} \| Z_{k + 1}^{( \ell )} \|\,,
        \end{equation}
        \noindent with the same $\tau_0 = ( \cstmseight \cstmseven \beta_{\max} )^{-1}$ as
        in~\eqref{eq:tau_zero_iteration}. Estimate~\eqref{eq:final_iter_error} is the assertion
        that the iterates converge to the solution at the same geometric rate at which they
        contract among themselves. The foregoing
        discussion establishes the following theorem.

		\begin{theorem}[One-step solvability and geometric convergence]\label{thm:iteration_convergence}
			Fix a time level $k$ and assume that the preceding iterates $u_{k-1}, u_k$ exist and
			satisfy the bounds~\eqref{eq:unif_bound_half_lu} and \eqref{eq:unif_bound_lu}, that is
			$\| \L^{1/2} u_j \| \leq \cstmtwo$ for $j = k-1,k$ and $\| \L u_k \| \leq \cstmfour$. Let
			$\tau_0 = \bigl( \cstmseight \cstmseven \beta_{\max} \bigr)^{-1}$ be as
			in~\eqref{eq:tau_zero_iteration}. It depends only on
			the data of problem \eqref{eq:main_kirchhoff}--\eqref{eq:initial_conds} through
			$\cstmtwo$, $\cstmfour$, $\alpha_{\min}$, $\beta_{\max}$, $T$ and $\max_i \| f_i \|$, and neither on the
			time level $k$, nor on the number $n$ of time steps, nor on $\tau$ itself. Then, for
			every $0 < \tau < \tau_0$, the nonlinear
			equation~\eqref{eq:nonlinear_equation_v} possesses exactly one solution $v_{k+1}$,
			and the iterative process \eqref{eq:oper_iter}
			converges to it at a geometric rate, with contraction factor $\rho = \tau / \tau_0 < 1$.
		\end{theorem}
		\begin{remark}[What this theorem does and does not give]\label[remark]{rem:one_step_only}
			\thmref{thm:iteration_convergence} concerns \emph{one} time level, and does not by
			itself give the whole discrete trajectory $\{ u_k \}_{k=0}^{n}$: the bounds it assumes
			are exactly what has to be propagated from one level to the next.
			\lemref{prop:lemma1_nonlinear} and \lemref{prop:lemma2_nonlinear} supply $\cstmtwo$
			and $\cstmfour$, and although they are stated for a complete solution sequence, their
			proofs are index-local: the energy identities used to bound $E_k$ and $\hat{E}_k$
			invoke the scheme only at the levels below $k$. Alternating them with the present
			theorem would therefore already construct a trajectory step by step, on the local
			interval $\left[ 0,\overline{T} \right]$ of \lemref{prop:lemma2_nonlinear} and under
			the data regularity~\eqref{eq:lemma2_data_regularity} that lemma requires. The companion paper~\cite{KirchhoffFEM} derives existence and uniqueness for the fully-discrete solution on the whole time interval $[0,T]$, where the solution is bounded. 
		\end{remark}

\section{Numerical Experiments}\label{sec:numerics}
\noindent We verify the four assertions of this paper that are open to direct numerical test:
that the locally linear scheme~\eqref{eq:linear_semidiscrete} conserves $\widetilde{E}_{1,k}$
and the nonlinear scheme~\eqref{eq:nonlinear_semidiscrete} conserves $\widetilde{E}_{2,k}$
exactly; that $E_k^{\ast}$ approximates the continuous energy to order~$\tau^2$
(\corref{cor:energy_approximation}); that both schemes converge to second order in time
(\thmref{prop:conver_thm_nonlinear}), including when $\alpha$ and $\beta$ genuinely depend on
$t$; and that the fixed-point iteration~\eqref{eq:oper_iter} converges geometrically
(\thmref{thm:iteration_convergence}).

\noindent\textbf{A spatially exact test problem.} All four statements concern the temporal
discretization alone, so we choose a setting in which the spatial discretization contributes
no error whatsoever. On the square $\domain = (0,\ell_\domain)^2$ the function
\begin{equation}\label{eq:eigenmode}
  \varphi
  = \sin \left( \tfrac{\lambda \pi x}{\ell_\domain} \right)
    \sin \left( \tfrac{\lambda \pi y}{\ell_\domain} \right)\,,
  \qquad \lambda \in \mathbb{N}\,,
\end{equation}
is an eigenfunction of the Dirichlet Laplacian, $\L \varphi = \tilde{\lambda} \varphi$ with
$\tilde{\lambda} = 2 ( \lambda \pi / \ell_\domain )^2$. Both schemes involve $u_{k\pm1}$ only
through $\L u_{k\pm1}$ and through the scalar $\dint{| \nabla u_{k\pm1} |^2}$, so each maps the
one-dimensional space $\operatorname{span}\{\varphi\}$ into itself. If the initial data
\eqref{eq:starting_vec1} lie in that span, we take
$u_0 = a \varphi$ and $v_0 = b \varphi$, then the entire discrete evolution does,
and writing $u_k = c_k \varphi$ reduces each scheme \emph{exactly} to a scalar three-layer
recursion for $c_k$. With $m_\varphi = \| \varphi \|^2$ and $G = \dint{| \nabla \varphi |^2} = \tilde{\lambda} m_\varphi$,
the homogeneous problem~\eqref{eq:homog_kirchhoff} becomes the Duffing-type equation
\begin{equation}\label{eq:scalar_reduction}
  c^{\prime\prime} ( t )
  + \tilde{\lambda} \left( \alpha + \beta G \, c^2 ( t ) \right) c ( t ) = 0\,,
  \qquad c ( 0 ) = a\,,\quad c^{\prime} ( 0 ) = b\,,
\end{equation}
whose conserved energy is $E = \tfrac12 m_\varphi \, (c^{\prime})^2 + \alpha V + \beta V^2$ with
$V = \tfrac12 G c^2$, in agreement with~\eqref{eq:total_energy}. The locally linear scheme
reduces to $c_{k+1} = 2 c_k / ( 1 + A_k ) - c_{k-1}$ with
$A_k = \tfrac12 \tilde{\lambda} \tau^2 ( \alpha + \beta G c_k^2 )$, and the nonlinear scheme to a scalar
cubic equation for $c_{k+1}$, which we solve by Newton's method to a tolerance of $10^{-15}$.

\noindent This reduction is not a simplification of the problem but a restriction of the data:
the nonlinearity is fully active, since $\dint{| \nabla u_k |^2} = G c_k^2$ varies over the
evolution, and no term of either scheme is switched off. In the reported runs $\ell_\domain = 1$, $\lambda = 1$, $\alpha = \beta = 1$, $a = 1$
and $b = 0$, so that $\tilde{\lambda} = 2\pi^2$, $m_\varphi = \tfrac14$, $G = \tfrac{\pi^2}{2}$ and
$E \equiv E(0) \approx 8.5554693$; the solution is periodic with period $\approx 0.661$,
which is shorter than the small-amplitude period $2\pi/\sqrt{\tilde{\lambda}} \approx 1.414$ because the
nonlinearity stiffens the restoring force. All computations are carried out in double
precision.

\noindent\textbf{Conservation of the discrete invariants.}
\Cref{fg:energy_drift} reports the relative deviation
$| \widetilde{E}_{i,k} - \widetilde{E}_{i,1} | / | \widetilde{E}_{i,1} |$ over
$200\,000$ time steps on $[0,100]$, that is over roughly $150$ periods, at
$\tau = 5 \cdot 10^{-4}$. Both invariants remain constant up to rounding errors that accumulate over
this many operations in double precision. The decisive observation is what happens under
refinement. Repeating the experiment on the same interval with an increased number of time steps increases rather than decreases the considered error. This is precisely the
signature of an identity rather than of an approximation.

\begin{figure}[!ht]
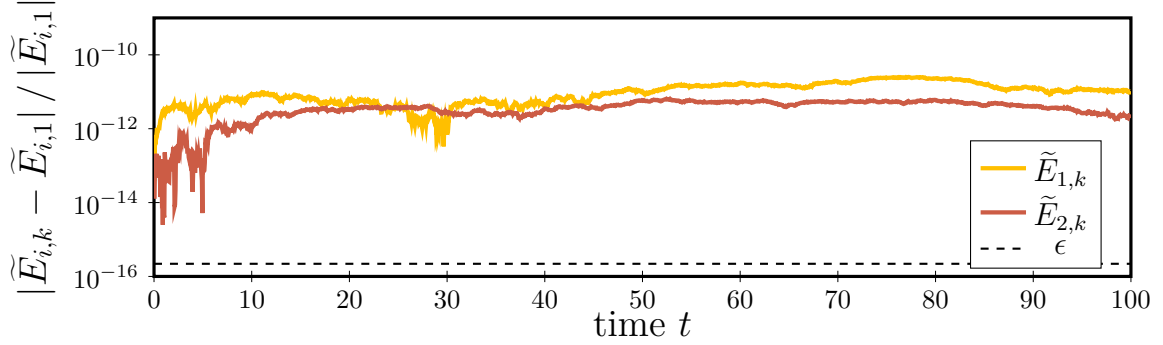

  \centering
  \includestandalone{energy_drift}
  \caption{Relative drift of the discrete invariants $\widetilde{E}_{1,k}$ (locally linear
  scheme) and $\widetilde{E}_{2,k}$ (nonlinear scheme) over $200\,000$ time steps on
  $[0,100]$, at $\tau = 5 \cdot 10^{-4}$, for $\alpha = \beta = 1$ and $\lambda = 1$. Both
  remain constant to within the accumulation of rounding errors.}\label{fg:energy_drift}
\end{figure}

\begin{figure}[!ht]
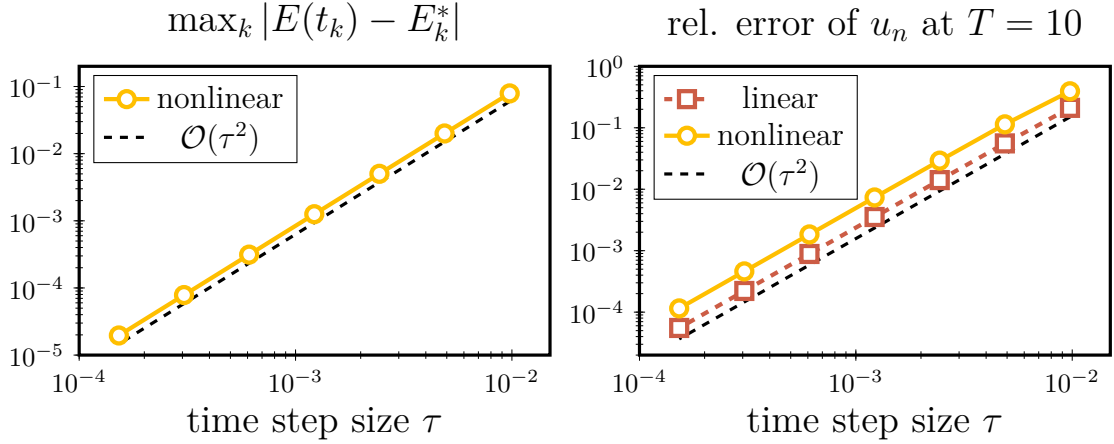

  \centering
  \includestandalone{energy_order}
  \caption{Left: maximal difference between the continuous energy $E(t_k)$ and its discrete
  approximation $E_k^{\ast}$ on $[0,10]$, against the time step size. Right: relative error of
  the two schemes at $T = 10$, measured against a high-accuracy reference solution. Both
  quantities decrease at the predicted order~$\tau^2$.}\label{fg:energy_order}
\end{figure}

\noindent\textbf{Approximation of the continuous energy and temporal order.} The left panel
of \Cref{fg:energy_order} reports $\max_k | E(t_k) - E_k^{\ast} |$ on $[0,10]$ against $\tau = 10 \cdot 2^{-j}$, $j = 10, \dots, 16$, confirming
\corref{cor:energy_approximation}. The right panel reports the relative error of $u_n$ at
$T = 10$ for both schemes, measured against a reference solution
of~\eqref{eq:scalar_reduction} computed by the classical fourth-order Runge--Kutta method with
$4 \cdot 10^{6}$ steps. Both schemes attain the second-order
convergence in time of \thmref{prop:conver_thm_nonlinear}.

\noindent It is worth recording that on this problem the locally linear scheme is the better of
the two on both counts; its error is smaller at every step
size tested, and it requires no nonlinear solve at all, since its coefficient $q_k$ is
evaluated at the known level $u_k$. We record this because it is what the experiment shows, and not as an argument for
preferring one scheme over the other; a single test problem fixes an error constant, not a
method. 

\FloatBarrier

\noindent\textbf{Geometric convergence of the fixed-point iteration.} 
\noindent \Cref{tab:fixed_point} reports the
first-step contraction ratio, \ie{}, the ratio of the second increment to the first, maximised over
the time levels, and the largest number of iterations needed to reduce the increment below
$10^{-14}$. The iteration converges geometrically at every
time level and from the analysed starting iterate $v_{k+1}^{(0)} = (u_k + u_{k-1})/2$, as
\thmref{thm:iteration_convergence} asserts. The observed factor, however, falls by almost
exactly a factor of four whenever $\tau$ is halved, so that in this example the iteration
contracts like $\bigO(\tau^2)$ rather than like the $\rho = \tau/\tau_0$ the theorem guarantees.
There is no contradiction since $\tau/\tau_0$ is an upper bound, but it shows that the bound
is not sharp here, and the reason is instructive. The single factor $\tau^{-1}$ that turns the
$\tau^2$ carried by $\T_{k,\ell}$ into the $\tau$ of \eqref{eq:final_iter} comes from the operator
bound $\| \T_{k,\ell}^{-1} \L^{1/2} \| \leq ( \tau \sqrt{2 \alpha_{\min}} )^{-1}$
of~\eqref{eq:operator_ineq}. That bound is sharp, but it is attained at
$\lambda = 2 ( \tau^2 q )^{-1}$, that is only in the high-frequency limit
$\lambda \sim \tau^{-2}$. For a \emph{fixed} eigenvalue $\lambda$ one has
$\| \T_{k,m}^{-1} \L^{1/2} \varphi \| = \sqrt{\lambda} \, \| \varphi \| ( 1 + \tfrac{\tau^2}{2}
q \lambda )^{-1} = \bigO(1)$ as $\tau \to 0$. The present experiment excites one low mode, so
the $\tau^{-1}$ never materialises, and the observed rate is the full $\bigO(\tau^2)$ of the
coefficient increment. A solution with content at frequencies $\lambda \sim \tau^{-2}$ would
approach the proved rate; the theorem must cover that case, and is therefore not improvable
without a frequency-dependent hypothesis.

\begin{table}[!ht]
  \centering
  \begin{tabular}{lccccccc}
    \toprule
    $j$
      & $7$ & $8$ & $9$ & $10$ & $11$ & $12$ & $13$ \\
    \midrule
    $\max_k \rho_k$
      & $3.62$e$-1$ & $1.24$e$-1$ & $3.53$e$-2$ & $9.16$e$-3$
      & $2.31$e$-3$ & $5.80$e$-4$ & $1.45$e$-4$ \\
    iterations
      & $30$ & $15$ & $10$ & $7$ & $6$ & $5$ & $4$ \\
    \bottomrule
  \end{tabular}
  \caption{Geometric convergence of the analysed fixed-point
  iteration~\eqref{eq:oper_iter}, at $\tau = 10 \cdot 2^{-j}$. Here $\rho_k$ is the
  \emph{first-step contraction ratio}
  $\| \L^{1/2} ( v_{k+1}^{(2)} - v_{k+1}^{(1)} ) \| /
   \| \L^{1/2} ( v_{k+1}^{(1)} - v_{k+1}^{(0)} ) \|$ at time level $k$, not a maximum over the
  iteration index. ``iterations'' is the largest number of steps needed at any time level to
  bring the increment below $10^{-14}$. The ratio decreases by a factor of about four per
  halving of $\tau$.}\label{tab:fixed_point}
\end{table}

\FloatBarrier

\noindent\textbf{Time-dependent coefficients.} The convergence result \thmref{prop:conver_thm_nonlinear} retains the time
dependence, and we test it directly. The reduction of the preceding paragraphs carries over
unchanged to the inhomogeneous nonautonomous problem: prescribing a manufactured solution
$u ( t ) = g(t) \varphi$ with $\varphi$ as in~\eqref{eq:eigenmode} and inserting it
into~\eqref{eq:main_kirchhoff} gives the source term $f = F(t) \varphi$ with
\begin{equation}\label{eq:manufactured_source}
  F ( t )
  = g^{\prime\prime} ( t )
  + \tilde{\lambda} \left( \alpha ( t )
    + \beta ( t ) G \, g^2 ( t ) \right) g ( t )\,,
\end{equation}
so that the exact coefficient is $g$ itself and no reference integrator is needed. We
consider 
\begin{alignat}{2}
  \alpha ( t ) &= 1 + 4t\,, &\quad
  \beta ( t ) &= 2 - 4t\,,\quad g ( t ) = e^{\pi t}\,,\quad
  T = \tfrac14\,,\label{eq:nonautonomous_B}
\end{alignat}
with $\lambda = 1$. The coefficients satisfy $\alpha \geq 1$ and $\beta \geq 1$ on the
respective intervals, as required. This example lets the solution grow
exponentially, which is the regime in which the locality in time of the analysis makes itself
felt, and its horizon is chosen accordingly. \Cref{fg:nonautonomous_order} reports the relative
error at $T$. Both schemes attain the optimal second-order rate. 

\begin{figure}[!ht]
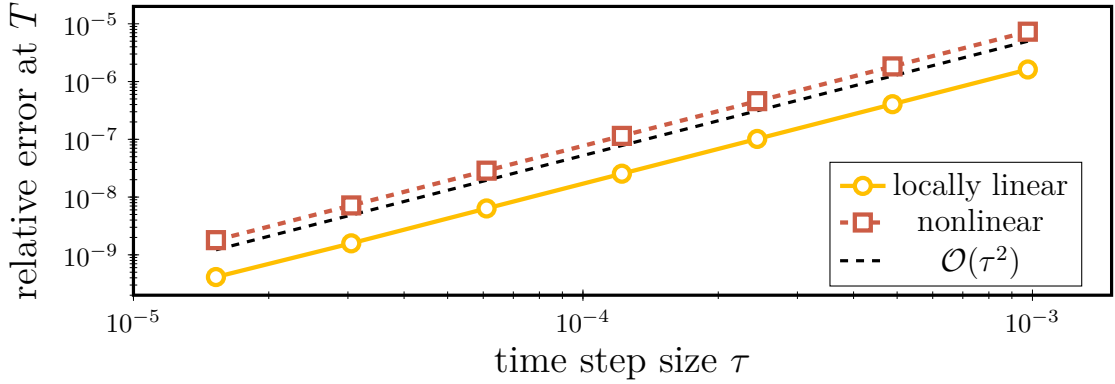

  \centering
  \includestandalone{nonautonomous_order}
  \caption{Relative error at the final time for the nonautonomous test
  case~\eqref{eq:nonautonomous_B}, \ie{}, $\alpha = 1+4t$, \ $\beta = 2-4t$, for both three-layer
  schemes. }\label{fg:nonautonomous_order}
\end{figure}

\noindent The script producing all data sets is
\texttt{/examples/run\_energy\_conservation.py} in the accompanying repository. It
depends only on the Python standard library, since by the reduction above no finite element
computation is required. Numerical experiments for the fully discrete method, in which the
spatial error is the object of study, are reported in the companion
paper~\cite{KirchhoffFEM}.

\section{Conclusion}\label{sec:conclusion}

\noindent We have studied conservative temporal discretizations of a two-dimensional
Kirchhoff-type integro-differential equation with time-dependent coefficients. The two
symmetric three-layer schemes of Crank--Nicolson type analysed preserve respective discrete energies
identically and at every time level. Neither conserves the continuous energy $E(t)$, which is
instead approximated to $\bigO(\tau^2)$ by the separate quantity $E_k^{\ast}$. For the nonlinear
scheme, we established uniform bounds on $(u_k - u_{k-1})/\tau$ and $\L^{1/2} u_k$ by working
directly with the discrete energies, and local uniform bounds on
$\L^{1/2}(u_k-u_{k-1})/\tau$ and $\L u_k$ and from them local second-order convergence in time, for the solution and for the
central-difference approximation of its first time derivative. The associated iteration to the nonlinear equation at each
time level converges to the unique solution geometrically
for $\tau$ small enough (\thmref{thm:iteration_convergence}). Combined with the index-local
\apriori{} bounds, this constructs the trajectory stepwise on the interval those bounds cover. The two conservation identities, the approximation of $E ( t_k )$ by $E_k^{\ast}$, the second order in time, and the geometric rate of the iteration are exhibited in \Cref{sec:numerics}.

\noindent Several questions remain open. A single result with existence and uniqueness of a discrete solution for the whole time interval $[0,T]$ until the solution blows up does not exist to our knowledge. Finally, the energy conservation established here for the homogeneous problem with
constant coefficients has no counterpart yet for time-dependent $\alpha$ and $\beta$, where one
expects a controlled drift rather than exact conservation. 
\noindent The companion paper~\cite{KirchhoffFEM} combines the time stepping analysed here
with a conforming finite element method and a Picard iteration for the nonlinearity, and
establishes convergence of the resulting fully discrete scheme.

%% file: figures/energy_drift.tex
\begin{tikzpicture}

\definecolor{darkgray176}{RGB}{176,176,176}
\definecolor{darkcoral}{rgb}{0.8, 0.36, 0.27}
\definecolor{amber}{rgb}{1.0, 0.75, 0.0}

\begin{axis}[
width=14.5cm,
height=5.cm,
axis line style=very thick,
log basis y={10},
tick align=center,
tick pos=left,
legend columns=1,
legend style={/tikz/column 2/.style={column sep=7pt,}, font=\large},
legend pos=south east,
xlabel={{\Large time $t$}},
x label style={at={(axis description cs:0.5,-0.1)},anchor=north},
x grid style={darkgray176},
xmin=0, xmax=100,
xtick style={color=black},
ylabel={{\Large $| \widetilde{E}_{i,k} - \widetilde{E}_{i,1} | \, / \, | \widetilde{E}_{i,1} |$}},
y label style={at={(axis description cs:-0.09,0.5)},anchor=south},
y grid style={darkgray176},
ymin=1e-16, ymax=1e-9,
ymode=log,
ytick style={color=black},
ytick={1e-16,1e-14,1e-12,1e-10},
yticklabels={
  \(\displaystyle {10^{-16}}\),
  \(\displaystyle {10^{-14}}\),
  \(\displaystyle {10^{-12}}\),
  \(\displaystyle {10^{-10}}\),
}
]

\addplot [amber, ultra thick, solid] table[x=t,y=drift_E1,col sep=comma]
  {./figures/energy_drift.csv};
\addlegendentry{$\widetilde{E}_{1,k}$}

\addplot [darkcoral, ultra thick, solid] table[x=t,y=drift_E2,col sep=comma]
  {./figures/energy_drift.csv};
\addlegendentry{$\widetilde{E}_{2,k}$}

\addplot [black, thick, dashed, no markers] coordinates {(0,2.2e-16) (100,2.2e-16)};
\addlegendentry{$\epsilon$}

\end{axis}

\end{tikzpicture}

%% file: figures/energy_order.tex
\begin{tikzpicture}

\definecolor{darkgray176}{RGB}{176,176,176}
\definecolor{darkcoral}{rgb}{0.8, 0.36, 0.27}
\definecolor{amber}{rgb}{1.0, 0.75, 0.0}

% ---------------------------------------------------------------- left panel
\begin{scope}[shift={(0cm,0cm)}]
\begin{axis}[
title={{\Large $\max_k | E(t_k) - E_k^{\ast} |$}},
width=7.8cm,
height=5.4cm,
axis line style=very thick,
log basis x={10},
log basis y={10},
tick align=center,
tick pos=left,
legend style={font=\large},
legend pos=north west,
xlabel={{\Large time step size $\tau$}},
x label style={at={(axis description cs:0.5,-0.15)},anchor=north},
x grid style={darkgray176},
xmin=1e-4, xmax=1.5e-2,
xmode=log,
xtick style={color=black},
y grid style={darkgray176},
ymin=1e-5, ymax=2e-1,
ymode=log,
ytick style={color=black},
]

\addplot [amber, ultra thick, solid, mark=*, mark size=3,
          every mark/.append style={fill=white,solid}]
  table[x=tau,y=err,col sep=comma]{./figures/energy_approximation.csv};
\addlegendentry{nonlinear}

\addplot [black, very thick, dashed, no markers] coordinates {
  (9.765625e-03, 6.0e-2) (1.52587890625e-04, 1.46484375e-05)
};
\addlegendentry{$\mathcal{O}(\tau^2)$}

\end{axis}
\end{scope}

% --------------------------------------------------------------- right panel
\begin{scope}[shift={(7.4cm,0cm)}]
\begin{axis}[
title={{\Large rel. error of $u_n$ at $T=10$}},
width=7.8cm,
height=5.4cm,
axis line style=very thick,
log basis x={10},
log basis y={10},
tick align=center,
tick pos=left,
legend style={font=\large},
legend pos=north west,
xlabel={{\Large time step size $\tau$}},
x label style={at={(axis description cs:0.5,-0.15)},anchor=north},
x grid style={darkgray176},
xmin=1e-4, xmax=1.5e-2,
xmode=log,
xtick style={color=black},
y grid style={darkgray176},
ymin=2e-5, ymax=1e0,
ymode=log,
ytick style={color=black},
]

\addplot [darkcoral, ultra thick, dashed, mark=square*, mark size=3,
          every mark/.append style={fill=white,solid}]
  table[x=tau,y=err_linear,col sep=comma]{./figures/temporal_order.csv};
\addlegendentry{linear}

\addplot [amber, ultra thick, solid, mark=*, mark size=3,
          every mark/.append style={fill=white,solid}]
  table[x=tau,y=err_nonlinear,col sep=comma]{./figures/temporal_order.csv};
\addlegendentry{nonlinear}

\addplot [black, very thick, dashed, no markers] coordinates {
  (9.765625e-03, 1.5e-1) (1.52587890625e-04, 3.662109375e-05)
};
\addlegendentry{$\mathcal{O}(\tau^2)$}

\end{axis}
\end{scope}

\end{tikzpicture}

%% file: figures/nonautonomous_order.tex
\begin{tikzpicture}

\definecolor{darkgray176}{RGB}{176,176,176}
\definecolor{darkcoral}{rgb}{0.8, 0.36, 0.27}
\definecolor{amber}{rgb}{1.0, 0.75, 0.0}

% ---------------------------------------------------------------- left panel
\begin{scope}[shift={(0cm,0cm)}]
\begin{axis}[
width=14.5cm,
height=5.4cm,
axis line style=very thick,
log basis x={10},
log basis y={10},
tick align=center,
tick pos=left,
legend style={font=\large},
legend pos=south east,
xlabel={{\Large time step size $\tau$}},
x label style={at={(axis description cs:0.5,-0.15)},anchor=north},
x grid style={darkgray176},
xmin=1e-5, xmax=1.5e-3,
xmode=log,
xtick style={color=black},
ylabel={{\Large relative error at $T$}},
y label style={at={(axis description cs:-0.09,0.5)},anchor=south},
y grid style={darkgray176},
ymin=2e-10, ymax=2e-5,
ymode=log,
ytick style={color=black},
]

\addplot [amber, ultra thick, solid, mark=*, mark size=3,
          every mark/.append style={fill=white,solid}]
  table[x=tau,y=err_linear,col sep=comma]{./figures/nonautonomous_B.csv};
\addlegendentry{locally linear}

\addplot [darkcoral, ultra thick, dashed, mark=square*, mark size=3,
          every mark/.append style={fill=white,solid}]
  table[x=tau,y=err_nonlinear,col sep=comma]{./figures/nonautonomous_B.csv};
\addlegendentry{nonlinear}

\addplot [black, very thick, dashed, no markers] coordinates {
  (9.765625e-04, 5.0e-6) (1.52587890625e-05, 1.220703125e-09)
};
\addlegendentry{$\mathcal{O}(\tau^2)$}

\end{axis}
\end{scope}

\end{tikzpicture}

%% file: common/backmatter.tex
% Shared back matter: acknowledgments, declarations, bibliography.
\section*{Declarations}
\ifblind
\noindent\textbf{Funding.} Funding information is withheld for double-blind review and
will be supplied with the camera-ready version.
\else
\noindent\textbf{Funding.} This work by A.~Rupp was supported by the Deutsche Forschungsgemeinschaft
(DFG, German Research Foundation) -- Project-ID 577175348. And in part supported by the Deutsche
Forschungsgemeinschaft under Germany's Excellence Strategy -- EXC-2047/2 -- 390685813, through
F.~Krumbiegel's and A.~Rupp's stay at the Hausdorff Research Institute for Mathematics. {Additionally, Z.~Vashakidze was supported by the Shota Rustaveli National Science Foundation of Georgia (SRNSFG) under Grant No. FR-25-215.} The authors declare no further funding for this work.
\fi

\noindent\textbf{Competing interests.} The authors declare that they have no competing
interests.

\noindent\textbf{Data and code availability.} \paperdataavailability
%

		% ===================== BIBLIOGRAPHY OUTPUT LOCATION ==========================
		% --- BIBLIOGRAPHY -----------------------------------------------------------
		% References section with visible heading + PDF bookmark
		% Requires: \usepackage{biblatex} (with \addbibresource{...}) and \usepackage{hyperref}
		
		% % === METHOD: Start bibliography on a clean page (auto oneside/twoside) ===
		% \makeatletter                                   % enable access to class internals (@ macros)
		% \if@twoside\cleardoublepage\else\clearpage\fi   % right-hand page for twoside, plain new page for oneside
		% \makeatother                                    % restore normal category code for @
		
		% --- METHOD: Ensure correct hyperlink anchor --------------------------------
		\phantomsection                % Ensure correct hyperlink anchor for references (used with hyperref)
		
		% === METHOD: Print Bibliography with visible heading =======================
		% BibLaTeX creates a section-like heading ("References") and prints entries.
		% With hyperref loaded, this normally adds one PDF bookmark automatically.
		\printbibliography[
		heading=bibliography,   % use biblatex’s built-in section-level heading
		title={References}      % visible heading text
		]                         % no manual bookmark needed in most setups
		
		% === METHOD: (Optional) Force-add a PDF bookmark ============================
		% Use ONLY if your class/template fails to produce the bookmark above.
		% Toggle the switch below to true when you need the manual bookmark.
		\newif\ifForceBibBookmark
		% \ForceBibBookmarkfalse    % default: off (prevents duplicate bookmarks)
		\ForceBibBookmarktrue     % turn on if class/template does not create bookmark
		
		\ifForceBibBookmark
		\phantomsection                  % uncomment if you need to adjust the anchor position
		\pdfbookmark[1]{References}{bib} % level-1 bookmark titled "References", anchor name "bib"
		\fi